\documentclass[11pt,reqno]{amsart}
\usepackage{amsmath,epsfig,graphicx,color}
\usepackage{subcaption}
\usepackage{float}
\usepackage{dsfont}
\usepackage{ifthen}
\usepackage{amssymb,latexsym}
\usepackage{makecell}
\usepackage{hyperref}

\definecolor{bordeau}{rgb}{0.5,0,0}
\definecolor{pslblue}{RGB}{36, 56, 141}
\hypersetup{
  urlcolor=black, 
  menucolor=black, 
  citecolor=bordeau, 
  anchorcolor=black, 
  filecolor=black, 
  linkcolor=pslblue, 
  colorlinks=true,
}

\usepackage[numbers,sort]{natbib} 
\newcommand{\cip}{\stackrel{\P}{\rightarrow}}
\newcommand{\eid}{\stackrel{\rm d}{=}}
\newcommand{\dint}{\,\mathrm{d}}

\newcommand{\y}{\mathbf{y}}
\newcommand{\X}{\mathbf{X}}
\newcommand{\D}{\mathbb{D}}
\newcommand{\Tr}{\operatorname{Tr}}

\usepackage{mathtools}
\mathtoolsset{showonlyrefs}

\usepackage{multirow}

\definecolor{darkblue}{rgb}{.2, 0.2,.8}
\definecolor{darkgreen}{rgb}{0,0.5,0.3}
\definecolor{darkred}{rgb}{.8, .1,.1}

\newtheorem{lemma}{Lemma}[section]
\newtheorem{theorem}[lemma]{Theorem}
\newtheorem{proposition}[lemma]{Proposition}
\theoremstyle{definition}
\newtheorem{definition}[lemma]{Definition}
\newtheorem{corollary}[lemma]{Corollary}
\newtheorem{example}[lemma]{Example}
\newtheorem{remark}[lemma]{Remark}

\newcommand{\cid}{\stackrel{\rm d}{\rightarrow}}

\newcommand{\sign}{{\rm sign}}

\newcommand{\as}{{\rm a.s.}}

\newcommand{\diag}{{\rm diag}}
\newcommand{\rhs}{right-hand side}

\newcommand{\N}{\mathbb{N}}
\newcommand{\R}{\mathbb{R}}

\newcommand{\nto}{n\to\infty}

\newcommand{\vep}{\varepsilon}

\newcommand{\bfI}{{\bf I}}

\newcommand{\E }{{\mathbb E}}
\renewcommand{\P }{{\mathbb P}}

\newcommand{\1}{\mathds{1}}

\allowdisplaybreaks

\DeclareMathOperator{\e}{e}
\newcommand{\Y}{{\mathbf Y}}
\newcommand{\DAN}{{\mathrm{DAN}}}
\newcommand{\C}{\mathbb{C}}

\renewcommand{\geq}{\geqslant}
\renewcommand{\leq}{\leqslant}
\renewcommand{\ge}{\geqslant}
\renewcommand{\le}{\leqslant}
\newcommand{\Prob}{\P}

\renewcommand{\bar}{\widebar}

\makeatletter
\let\save@mathaccent\mathaccent
\newcommand*\if@single[3]{%
  \setbox0\hbox{${\mathaccent"0362{#1}}^H$}%
  \setbox2\hbox{${\mathaccent"0362{\kern0pt#1}}^H$}%
  \ifdim\ht0=\ht2 #3\else #2\fi
  }
\newcommand*\rel@kern[1]{\kern#1\dimexpr\macc@kerna}
\newcommand*\widebar[1]{\@ifnextchar^{{\wide@bar{#1}{0}}}{\wide@bar{#1}{1}}}
\newcommand*\wide@bar[2]{\if@single{#1}{\wide@bar@{#1}{#2}{1}}{\wide@bar@{#1}{#2}{2}}}
\newcommand*\wide@bar@[3]{%
  \begingroup
  \def\mathaccent##1##2{%
    \let\mathaccent\save@mathaccent
    \if#32 \let\macc@nucleus\first@char \fi
    \setbox\z@\hbox{$\macc@style{\macc@nucleus}_{}$}%
    \setbox\tw@\hbox{$\macc@style{\macc@nucleus}{}_{}$}%
    \dimen@\wd\tw@
    \advance\dimen@-\wd\z@
    \divide\dimen@ 3
    \@tempdima\wd\tw@
    \advance\@tempdima-\scriptspace
    \divide\@tempdima 10
    \advance\dimen@-\@tempdima
    \ifdim\dimen@>\z@ \dimen@0pt\fi
    \rel@kern{0.6}\kern-\dimen@
    \if#31
      \overline{\rel@kern{-0.6}\kern\dimen@\macc@nucleus\rel@kern{0.4}\kern\dimen@}%
      \advance\dimen@0.4\dimexpr\macc@kerna
      \let\final@kern#2%
      \ifdim\dimen@<\z@ \let\final@kern1\fi
      \if\final@kern1 \kern-\dimen@\fi
    \else
      \overline{\rel@kern{-0.6}\kern\dimen@#1}%
    \fi
  }%
  \macc@depth\@ne
  \let\math@bgroup\@empty \let\math@egroup\macc@set@skewchar
  \mathsurround\z@ \frozen@everymath{\mathgroup\macc@group\relax}%
  \macc@set@skewchar\relax
  \let\mathaccentV\macc@nested@a
  \if#31
    \macc@nested@a\relax111{#1}%
  \else
    \def\gobble@till@marker##1\endmarker{}%
    \futurelet\first@char\gobble@till@marker#1\endmarker
    \ifcat\noexpand\first@char A\else
      \def\first@char{}%
    \fi
    \macc@nested@a\relax111{\first@char}%
  \fi
  \endgroup
}
\makeatother

\begin{document}
\bibliographystyle{acm}
\title[characteristic polynomial of self-normalized random matrices]{characteristic polynomial of self-normalized random matrices}
\thanks{Johannes Heiny's and Xuechun Hu's research was partially supported by the Swedish Research Council via VR grant VR-2023-03577 ``High-dimensional extremes and random matrix structures''.}

\author[Q. François]{Quentin François}
\address{Univ. Lille, CNRS UMR 9189 – CRIStAL, 59651 Villeneuve d’Ascq, France}
\email{quentin.francois@univ-lille.fr}
\author[J. Heiny]{Johannes Heiny}
\address{Department of Mathematics, KTH Royal Institute of Technology, Stockholm, Sweden}
\email{heiny@kth.se}
\author[X. Hu]{Xuechun Hu} 
\address{Department of Mathematics, KTH Royal Institute of Technology, Stockholm, Sweden}
\email{xuechunh@kth.se}

\begin{abstract}
We study the characteristic polynomial of self-normalized random matrices, whose rows are independent and normalized to have unit $\mathrm{L}^2$ norm. The entries before normalization are assumed to have regularly varying tails with tail index $\alpha \in [0,2]$. We prove that, outside the unit disk, the characteristic polynomial converges to a random analytic function $F_\alpha$. We identify $F_\alpha$ as a multiplicative chaos described in terms of Poisson point processes. The family of limiting functions $(F_\alpha)_{\alpha\in[0,2]}$ interpolates between two universal regimes: Poisson multiplicative chaos at $\alpha=0$ and Gaussian multiplicative chaos at the boundary $\alpha=2$. Thus, self-normalization provides a matrix model where one observes the transition between Poissonian and Gaussian regimes for the limiting characteristic polynomial as the tail index varies. A similar transition is found for the fluctuations of the traces of self-normalized matrices. As an application of our results, we derive that the spectral radius of self-normalized matrices is asymptotically bounded above by one in probability for any symmetric entry distribution.
\end{abstract}

\maketitle

\section{Introduction}

The characteristic polynomial of a random matrix naturally encodes its spectral information in the form of a random analytic function. Complementary to the study of the eigenvalue point process, or the spectral measure, the convergence of the characteristic polynomial towards a random limit function can yield further results on the asymptotic spectral behavior of large matrices \cite{Basak_Zeitouni}. For instance, for matrices with independent and identically distributed (i.i.d.)\ entries, \cite{Tao_Vu} established the celebrated circular law which states that the spectral distribution converges towards the uniform measure on the unit disk under a finite second moment condition. Later, the convergence of the spectral radius to one, that is, the absence of outliers, was proved using the convergence of the characteristic polynomial \cite{Bordenave_Chafai_Garcia}. Related results have been obtained for random matrices with a general variance profile and for non-backtracking matrices of the configuration model~\cite{Hachem_Louvaris_2026,Louvaris_Wise_Yehuda_configuration}.

Limits of characteristic polynomials are well described in terms of multiplicative chaos which are random distributions of independent interest in mathematical physics, probability theory and random fields \cite{Bailey_Keating}. The convergence towards multiplicative chaos has been established for a large variety of random matrix models. The Gaussian multiplicative chaos appears as the limiting characteristic polynomial for centered i.i.d.\ matrices under a universal second moment condition \cite{Bordenave_Chafai_Garcia}, for circular $\beta$-ensembles \cite{Najnudel_Paquette_Simm, Chhaibi_Najnudel} and for adjacency matrices of Erdős-Rényi graphs \cite{Coste_Bernoulli} in a certain regime. Its Poissonnian analog, the Poisson multiplicative chaos, appears as the limiting characteristic polynomial for adjacency matrices of directed graphs given by a sum of uniform random permutation matrices \cite{Coste_Lambert_Zhu} and for Ewens random permutations \cite{franccois2025}.

In this paper, we consider the characteristic polynomial of \textit{self-normalized} matrices $\Y_n$ of the form 
\begin{equation}\label{def:Y}
    \Y_n= \left\{\diag \left(\X_n\X_n^\top \right) \right\}^{-1/2} \,\X_n\,,
\end{equation}
where $\X_n=(X_{ij})_{i,j\in [n]}$ is the \textit{data matrix} with i.i.d.\ entries and where for a matrix $A$, we denote by $\diag(A)$ the diagonal matrix having the same diagonal as $A$. Thus, the rows of $\Y_n$ are independent and have unit $\mathrm{L}^2$ norm. Self-normalized matrices play an important role in multivariate analysis since the \textit{sample correlation matrix} associated to the data matrix $\X_n$ is given by $\Y_n \Y_n^\top$. Several properties of large sample correlation matrices have recently been investigated: e.g., spectral distributions \cite{elkaroui:2009, doernemann:heiny:2025, heiny:yao:2022, heiny:2022}, linear spectral statistics \cite{jiang:pham:2025uniformity, Gao2017, yin:zheng:zou:2023, heiny:parolya:2024} and largest off-diagonal entries \cite{jiang:pham:2025asymptotic, jiang:pham:2025detecting, jiang:pham:2026largest, heiny:mikosch:yslas:2021}.

$\mathrm{L}^2$ normalization not only has nice geometric and probabilistic features (e.g., the rows of $\Y_n$ are uniformly distributed on the Euclidean unit sphere in the Gaussian case), it also brings statistical advantages since the spectral behavior of $\Y_n$ is more robust than the one of $\X_n$ with respect to heavy-tailed entries and dimension. The spectrum of the $\mathrm{L}^1$ normalized matrices, yielding random Markov matrices, was studied in \cite{Bordenave_Caputo_Chafai_reversible_Markov} for the reversible case and \cite{Bordenave_Caputo_Chafai_Piras} for the non-reversible case under an additional distribution dependent renormalization.
\medskip

Our main result is the convergence of the characteristic polynomial of self-normalized matrices $\Y_n$ when the entry distribution of the data matrix is symmetric and regularly varying with index $\alpha \in [0, 2]$.  We prove that, outside of the unit disk, the characteristic polynomial converges to a random analytic function $F_\alpha$. The limiting random analytic function $F_\alpha$ is a multiplicative chaos with a description in terms of Poisson point processes. Moreover, the family $(F_\alpha)_{\alpha \in [0,2]}$ interpolates between two universal types of limiting characteristic polynomials: Poisson multiplicative chaos at $\alpha=0$ and Gaussian multiplicative chaos at the boundary $\alpha=2$. In this sense, self-normalization provides a natural framework in which the transition between Poissonian and Gaussian spectral fluctuations can be observed as the tail index varies. As a consequence of the convergence result, we derive an upper bound on the spectral radius of self-normalized matrices. 

\subsection*{The model}

For a field $\{X_{ij};i,j\ge 1\}$ of i.i.d., non-degenerated, symmetrically distributed (i.e. $X_{ij} \eid -X_{ij}$) random variables, we construct the \textit{data matrix} $\X_n=(X_{ij})_{1\le i\le p; 1\le j \le n}$ and define the self-normalized matrix $\Y_n$ as in \eqref{def:Y} whose entries depend on $n$ and are given by 
\begin{equation}\label{def:R}
    Y_{ij}=Y_{ij}^{(n)}=\frac{X_{ij}}{\sqrt{X_{i1}^2+\cdots+X_{in}^2}}\,, \qquad i,j \in [n]\,.
\end{equation}
Throughout the paper, we often suppress the dependence on $n$ in our notation. Since $Y_{ij}$ is invariant with respect to a scaling of the $X_{ij}$'s, we will assume without loss of generality that $\E[X_{11}^2]=1$ whenever $\E[X_{11}^2]$ is finite.

In this paper, we will often work with $|X_{11}|$ having a regularly varying tail with index $\alpha\ge 0$, that is 
\begin{equation}\label{eq:regvar}
\P(|X_{11}|>x)=   x^{-\alpha}\, L(x)\,,\qquad x>0\,,
\end{equation}
for a function $L$ that is slowly varying at infinity that is, such that $\lim_{x\to \infty} L(tx)/L(x)=1$ for any $t>0$. Thus, regularly varying distributions possess power-law tails and moments of $|X_{11}|$ of higher order than $\alpha$ are infinite. Typical examples include the Pareto distribution with parameter $\alpha$ and the $t$-distribution with $\alpha$ degrees of freedom. Figure~\ref{fig:eigenvalue-clouds} shows realizations of the eigenvalues of $\Y_n$ for different values of $\alpha$.

\begin{figure}
    \centering
    \includegraphics[width=0.9\linewidth]{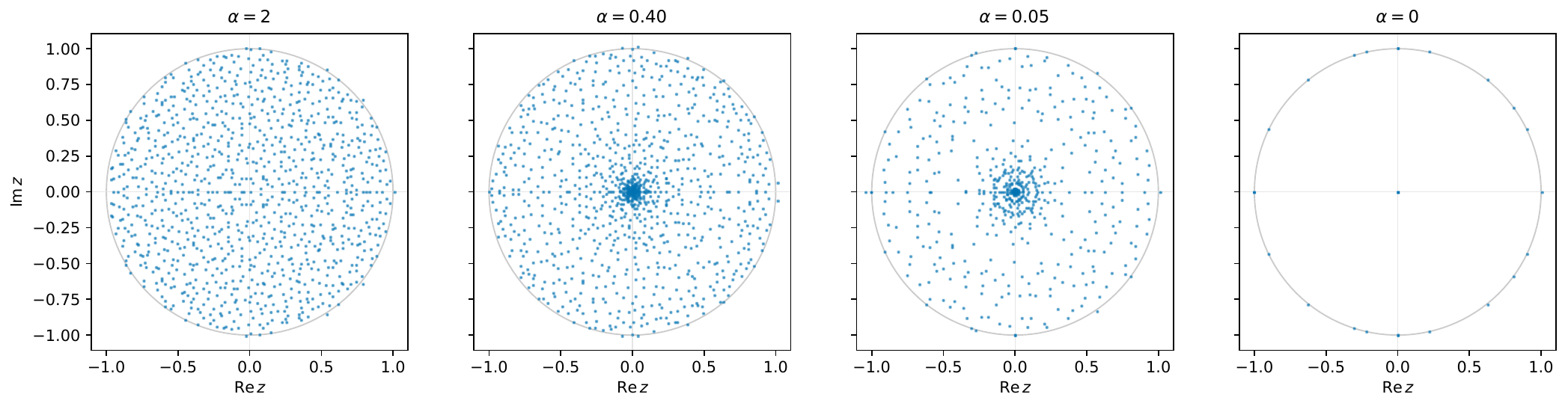}
    \caption{Eigenvalues of $\Y_n$ with $n=1000$.  For $\alpha\in\{2,0.40,0.05\}$, the entries are symmetric Pareto variables satisfying $\P(|X_{ij}|>x)=x^{-\alpha}$ for $x\ge1$. For $\alpha=0$, the entries have the slowly varying tail $\P(|X_{ij}|>x)=(\log x)^{-0.25}$ for $x\ge e$. The gray circle is the unit circle.}
    \label{fig:eigenvalue-clouds}
\end{figure}

\subsection*{Notation}
Convergence in distribution (resp.\ in probability) is denoted by $\cid$ (resp.\ $\cip$), equality in distribution by $\eid$, and unless explicitly stated otherwise all limits are for $\nto$. For sequences $(a_n)_n$ and $(b_n)_n$ we write $a_n=O(b_n)$ if $a_n/b_n\leq C$ for some constant $C>0$ and every $n\in\N$, and $a_n=o(b_n)$ if $\lim_{n\to\infty} a_n/b_n=0$. We write $a_n \lesssim b_n$ if there exists a positive constant $C$ not depending on $n$ such that $a_n \le C\, b_n$ for sufficiently large $n$. We denote by $\D = \left\{ z \in \C : |z| < 1 \right\}$ the open unit disk. For a random variable $X$, we write $X \in \DAN$ if the law of $X$ is in the domain of attraction of the normal distribution. 

\subsection*{Organization of the paper}

Section~\ref{sec:main_results} contains our main results. We introduce the notion of $\alpha$-Poisson families needed for the definition of the limiting random analytic function in Section \ref{subsec:alpha_Poisson} and state our main result on the convergence of the characteristic polynomial of self-normalized matrices in Theorem \ref{th:conv_carac_poly} in Section \ref{subsec:conv_carac_poly}. Theorem \ref{thm:conv_traces} establishes the convergence of traces of self-normalized matrices whose limiting distributions interpolate between Poissonian and Gaussian behaviors. Section \ref{subsec:boundary_results} presents the convergence results in the two boundary cases $\alpha = 0$ and $\alpha \geqslant 2$. The proof of Theorem~\ref{th:conv_carac_poly} is given in Section~\ref{sec:proof_main_thm}, while Section~\ref{sec:conv_traces} contains the proof of Theorem \ref{thm:conv_traces}. The proofs for the boundary cases are presented in Section~\ref{sec:limit_cases}.

\section{Main results}\setcounter{equation}{0}
\label{sec:main_results}

Our main result is the convergence of the characteristic polynomial of the self-normalized matrix~$\Y_n$,
\begin{equation}
\label{eq:char_poly}
    p_n(z) = \det(\bfI - z\Y_n), \qquad z \in \D \ ,
\end{equation}
to a random analytic function. The limiting function is defined using a family of Poisson processes called $\alpha$-Poisson family that is introduced in Section \ref{subsec:alpha_Poisson}. The convergence result is then stated in Section \ref{subsec:conv_carac_poly}. We also provide a non-Gaussian limit theorem for the traces of $\Y_n$. This reveals a stark contrast to the light-tailed case where the limit distribution is Gaussian.

\subsection{$\alpha$-Poisson families}
\label{subsec:alpha_Poisson}

The aim of this subsection is to introduce \textit{$\alpha$-Poisson families} which are essential for the construction of the limit of the characteristic polynomial of $\Y_n$. They will be defined using Poisson processes that we call \textit{$q$-Poisson processes}.

To this end, fix  $\alpha \in (0, 2)$ and let us first recall that if $U \sim \operatorname{Beta} \left(1- \tfrac{\alpha}{2}, \tfrac{\alpha}{2} \right)$ is a Beta random variable, then
\begin{equation}
    \label{eq:moments_beta}
    \E \left[ U^{m-1} \right] = \frac{\Gamma(m-\alpha/2)}{\Gamma(1-\alpha/2) \Gamma(m)} \,, \qquad m\ge 1 \, .
\end{equation}

\begin{definition}[$q$-Poisson process]
\label{def:q_poisson_process}
    Let $\alpha \in (0, 2)$ and let $(U_r)_{r \geqslant 1}$ be a family of 
    i.i.d. $\operatorname{Beta}(1 - \alpha/2, \alpha/2)$ variables. 
    For $q \geqslant 1$, set $W_q = \prod_{r=1}^q U_r$ and denote its law by $\mu_q^{(\alpha)}$. Let us define the measure $\nu_q^{(\alpha)}$ on $[-1, 1]$ by   
    \begin{equation}
    \nu_q^{(\alpha)}(\dint x) \coloneqq  \frac{1}{2q|x|} \mu_q^{(\alpha)}(\dint  |x|) \ ,
    \end{equation}
    that is, the measure for which for every $f: \R \to \R$ continuous and bounded,
\begin{equation}
    \int f(u) \nu_q^{(\alpha)}(\dint x) = 
    \frac{1}{2 q} \int_0^1 \frac{1}{u} (f(u) + f(-u)) \mu_q^{(\alpha)}(\dint u) \ .
\end{equation}
    The \textit{$q$-Poisson process} is the Poisson point process $N_q^{(\alpha)} = \sum_{i\ge 1} \delta_{\xi_{q, i}^{(\alpha)}}$ with intensity measure $\nu_q^{(\alpha)}$. 
\end{definition}

For $m \geqslant 1$ and $u \in [-1, 1]$, let us define the functions
\begin{equation}
    f_m(u) \coloneqq  \sign(u)^m |u|^{m/2} \, .
\end{equation}

\begin{definition}[$\alpha$-Poisson family]
\label{def:alpha_family}
    Let $\alpha \in (0, 2)$ and let 
    $\left(N_q^{(\alpha)}\right)_{q \geqslant 1}$ be independent $q$-Poisson processes defined in Definition \ref{def:q_poisson_process}. Then, for $q \geqslant 1$ and $m \geqslant 2$, let us set 
\begin{equation}
    \label{eq:def_V_q_m}
    V_{q, m}^{(\alpha)} \coloneqq  q \int f_m(u)  N_q^{(\alpha)}(\dint u) = q \sum_{i \geqslant 1} f_m \left(\xi_{q,i}^{(\alpha)} \right) \,.
\end{equation}
For $m=1$, since $f_1 \notin \mathrm{L}^1(\nu_q^{(\alpha)})$ but $f_1 \in \mathrm{L}^2(\nu_q^{(\alpha)})$, we define $V_{q, 1}^{(\alpha)}$ as a compensated Poisson integral
\begin{equation*}
    V_{q, 1}^{(\alpha)} \coloneqq  q \int f_1(u)  \tilde{N}_q^{(\alpha)}(\dint u) \in \mathrm{L}^2(\Prob)  \,,
\end{equation*}
where $\tilde N_q^{(\alpha)}(\dint u)
    =N_q^{(\alpha)}(\dint u)-\nu_q^{(\alpha)}(\dint u)$.
 We call an \textit{$\alpha$-Poisson family} the family $\left(Z_k^{(\alpha)} \right)_{k \geqslant 1}$ given by 
\begin{equation}
    \label{eq:family_alpha}
    Z_k^{(\alpha)} \coloneqq  \sum_{q | k} V_{q, k/q}^{(\alpha)} \, , \qquad k \geqslant 1 \,,
\end{equation}
where the summation in \eqref{eq:family_alpha} is over all divisors of $k$.
\end{definition}

Note that for any $m \geqslant 2$, 
\begin{equation}\label{eq:bound234}
    \E \left|V^{(\alpha)}_{q, m} \right| \leqslant q \int |f_m(\xi)| \dint \nu_q^{(\alpha)}(\xi) = \left(\E\big[ U^{m/2-1} \big]\right)^q \leqslant 1 \,.
\end{equation}
Therefore, the variables $V^{(\alpha)}_{q, m}$ are almost surely finite and so are the variables $Z^{(\alpha)}_{k}$. 

\subsection{Convergence of the characteristic polynomial and traces for $\alpha\in (0,2)$}
\label{subsec:conv_carac_poly}

Our main result is the convergence of the characteristic polynomial \eqref{eq:char_poly} towards a random analytic function defined in terms of $\alpha$-Poisson families. As in \cite{Bordenave_Chafai_Garcia, Chhaibi_Najnudel, Coste_Lambert_Zhu}, we consider the \textit{reciprocal} characteristic polynomial $p_n(z) = \det(\bfI - z\Y_n)$ which is related to the characteristic polynomial by $z^n p_n(1/z) = \det(z\bfI - \Y_n)$. Therefore, for $z \in \D$ we consider the usual characteristic polynomial in the region outside the unit disk. We endow the space of holomorphic functions on $\D$ with the topology of uniform convergence on compact sets. 

Let $\big( Z_k^{(\alpha)} \big)_{k\ge 1}$ be an $\alpha$-Poisson family. In view of \eqref{eq:bound234}, one has $\E \big| Z_k^{(\alpha)} \big| \leqslant k$. Thus, for any $\eta > 0$, an application of Markov's inequality yields $\Prob[|Z_k^{(\alpha)}| > e^{\eta k}] \leqslant k e^{-\eta k}$, which implies by a Borel-Cantelli argument that almost surely
\begin{equation*}
    \limsup_{k \to \infty} |Z_k^{(\alpha)}|^{1/k} \leqslant e^\eta \,.
\end{equation*}
Since the latter is valid for every $\eta > 0$, one derives that 
$\limsup_{k \to \infty} |Z_k^{(\alpha)}|^{1/k} \leqslant 1$ \as, so that the convergence radius of the random series 
\begin{equation*}
    \sum_{k \geqslant 1} \frac{z^k}{k} Z_k^{(\alpha)} 
\end{equation*}
is almost surely at least one by Hadamard's formula. 

\begin{theorem}[Convergence of the characteristic polynomial]
    \label{th:conv_carac_poly}
    Let $X_{11}$ be regularly varying with index $\alpha\in (0,2)$ and consider the characteristic polynomial $p_n(z) := \det (\bfI - z\Y_n)$. Then, we have the convergence in distribution, for the topology of local uniform convergence in $\D$
    \begin{equation}
        \label{eq:conv_carac_poly}
        p_n \cid F_\alpha \,, \qquad \nto\,,
    \end{equation}
    where $F_\alpha$ is the random holomorphic function 
    \begin{equation}
        \label{eq:expression_F}
        F_\alpha(z) = \exp \left( - \sum_{k \geqslant 1} \frac{z^k}{k} Z_k^{(\alpha)} \right) ,
    \end{equation}
and  
$\big(Z_k^{(\alpha)} \big)_{k \geqslant 1}$ is an $\alpha$-Poisson family from Definition~\ref{def:alpha_family}.
\end{theorem}

\begin{definition}[$\alpha$-heavy multiplicative chaos]
    We call the random analytic function $F_\alpha$ in \eqref{eq:expression_F} defined for $z \in \D$ the \textit{$\alpha$-heavy multiplicative chaos}.
\end{definition}

\begin{remark}[Gaussian to Poisson interpolation]
    The random analytic function $F_\alpha$ has the form of a \textit{multiplicative chaos}, that is, the exponential of a random series. In the limit case $\alpha \to 2$ or when $\alpha>2$, one finds that $p_n$ converges to the \textit{Gaussian multiplicative chaos}, see Section \ref{subsec:gaussian_case}, which is a universal object that was obtained as the limit for the characteristic polynomial of i.i.d.\ matrices with a finite second moment \cite{Bordenave_Chafai_Garcia}, for Gaussian elliptic matrices \cite{franois2306asymptotic} and for circular $\beta$-ensembles \cite{Najnudel_Paquette_Simm, Chhaibi_Najnudel}. 
    On the other side, for $\alpha=0$, the limiting function is a Poisson series, see Section \ref{subsec:alpha_0}, known as the Poisson multiplicative chaos, which in turn appeared as the limiting characteristic polynomial for uniform \cite{Coste_Lambert_Zhu} or Ewens random permutation matrices \cite{franccois2025}. 
     The limiting function $F_\alpha$ of Theorem \ref{th:conv_carac_poly} can therefore be seen as an interpolation between Gaussian and Poisson limits for characteristic polynomials. 
     \end{remark}
 
We have the following consequence of Theorem \ref{th:conv_carac_poly} for the spectral radius of $\Y_n$,
$$\rho(\Y_n):=\max_{i=1,\ldots,n} |\lambda_i(\Y_n)|\,,$$
where $\lambda_1(\Y_n),\ldots,\lambda_n(\Y_n)$ are the eigenvalues of $\Y_n$.

\begin{corollary}[Upper bound on the spectral radius]
\label{cor:upper_bound_spectral_radius}
    For every $\vep > 0$, it holds
    \begin{equation}
        \lim_{n \to \infty} \Prob[\rho(\Y_n) \geqslant 1 + \vep] = 0 \,.
    \end{equation}
\end{corollary}

\begin{proof}
    Let us fix $\vep > 0$. Using the continuous mapping theorem, 
    \begin{align*}
        \Prob \left[\rho(\Y_n) \geqslant 1 + \vep \right] &\leqslant 
        \Prob \left[\min_{|z| \le \tfrac{1}{1+\vep}} |\det(\bfI-z\Y_n)| = 0 \right] 
        \to \Prob \left[\min_{|z| \le \tfrac{1}{1+\vep}} |F_\alpha(z)| = 0 \right] = 0 \,.
    \end{align*}
\end{proof}

The finite-dimensional behavior of the spectral radius is illustrated in Figure~\ref{fig:pho}, where the empirical spectral radii remain close to the asymptotic upper bound $1$ across the displayed values of $\alpha$ and $n$.

\begin{figure}[h]
    \centering
    \includegraphics[width=0.5\linewidth]{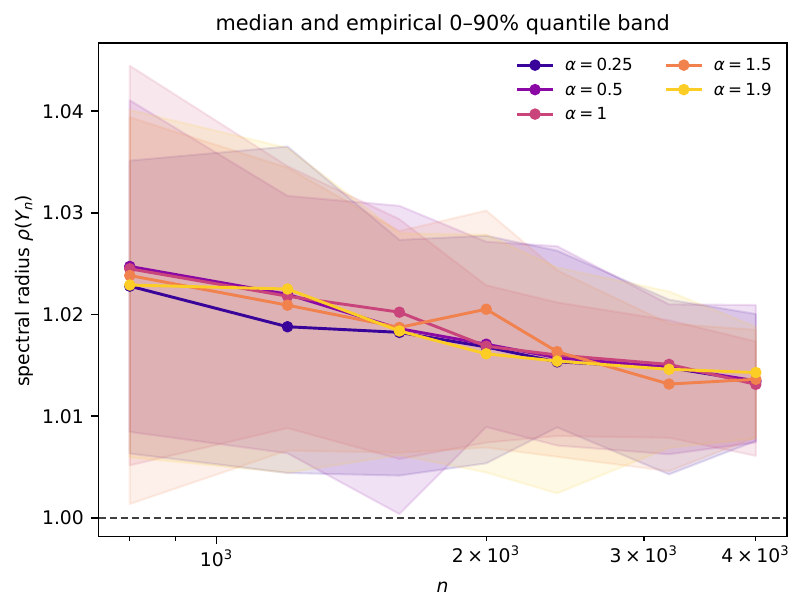}
    \caption{Illustration of Corollary ~\ref{cor:upper_bound_spectral_radius}. Spectral-radius simulations for $\Y_n$ with independent symmetric Pareto entries satisfying $\mathbb P(|X_{ij}|>x)=x^{-\alpha}$ for $x\geq1$. We use $\alpha\in\{0.25,0.5,1,1.5,1.9\}$ and $n\in\{800,1200,1600,2000,2400,3200,4000\}$.} 
    \label{fig:pho}
\end{figure}

Our second main result is the joint convergence of traces of the self-normalized matrix $\Y_n$. 

\begin{theorem}[Convergence of traces]
\label{thm:conv_traces}
    Let $X_{11}$ be regularly varying with index $\alpha\in (0,2)$ and let $\ell \geqslant 1$ and $k_1, \dots, k_{\ell} \geqslant 1$. Then we have the convergence in distribution 
    \begin{equation}
    \label{eq:conv_traces}
    \left(\Tr \big[ \Y_n^{k_1} \big], \dots, \Tr \big[ \Y_n^{k_\ell} \big] \right) \cid \left(Z_{k_1}^{(\alpha)}, \dots, Z_{k_\ell}^{(\alpha)} \right) \,, \qquad \nto\,,
\end{equation}
where $\big(Z_k^{(\alpha)}\big)_{k \geqslant 1}$ is an $\alpha$-family defined in Definition  \ref{def:alpha_family}.
\end{theorem}
It is worth emphasizing that the limits of the traces are non-Gaussian. This is in stark contrast to the literature on matrices with i.i.d.\ entries \cite{Bordenave_Chafai_Garcia} and circular $\beta$-ensembles \cite{Najnudel_Paquette_Simm}. Poisson limits were found for traces of (sums of) uniform permutation matrices \cite{Coste_Lambert_Zhu} and matrices with non-centered Bernoulli entries \cite{Coste_Bernoulli}. The convergence of traces was extended beyond the uniform case for random permutations following the Ewens distribution \cite{Nikeghbali_Zeindler}. 

Figure~\ref{fig:trace-interior} illustrates the trace convergence in Theorem~\ref{thm:conv_traces} for $\Tr \big[ \Y_n^{k} \big]$ with $k\in\{3,4,6\}$. For small $\alpha$, the limiting $\alpha$-Poisson laws retain a pronounced Poissonian structure, reflecting the contribution of a few exceptionally large entries. As $\alpha$ increases, the accumulation of many smaller contributions smooths the distribution, and the $\alpha$-Poisson law approaches the Gaussian endpoint as $\alpha\uparrow2$. For every fixed $\alpha<2$, however, the limiting law remains the $\alpha$-Poisson law of Theorem~\ref{thm:conv_traces}, rather than a Gaussian law.

\begin{figure}
    \centering
    \includegraphics[width=0.9\linewidth]{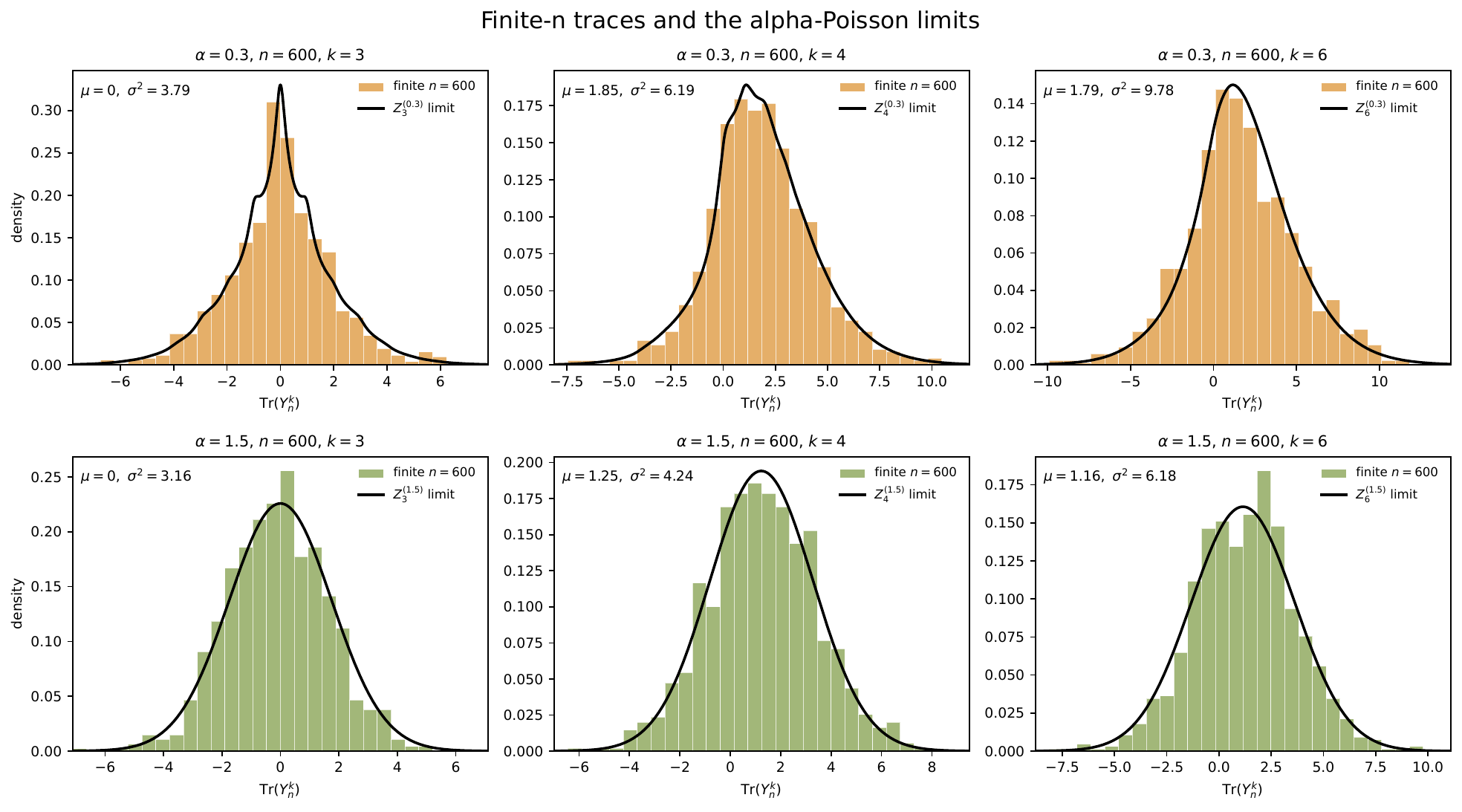}
    \caption{Illustration of Theorem \ref{thm:conv_traces}. Raw finite-$n$ traces for row-normalized symmetric Pareto entries satisfying $\mathbb P(|X_{11}|>x)=x^{-\alpha}$ for $x\geq1$, overlaid with the limiting marginal densities of $Z_k^{(\alpha)}$. The upper row uses $\alpha=0.3$ and illustrates behavior near the Poisson endpoint, whereas the lower row uses $\alpha=1.5$ and illustrates the approach toward the Gaussian endpoint. Each black curve is the corresponding $\alpha$-Poisson limit, not a Gaussian approximation.}
    \label{fig:trace-interior}
\end{figure}
\subsection{Boundary cases $\alpha=0$ and $\alpha=2$}\label{subsec:boundary_results}

Interestingly, the limits in Theorems~\ref{th:conv_carac_poly} and \ref{thm:conv_traces} interpolate between Poisson behavior at $\alpha=0$ and Gaussian behavior at $\alpha=2$.

\subsubsection{Poisson behavior}\label{sec:poissonbehavior}

Recall that for $\alpha\in (0,2)$ the random variables $Z_k^{(\alpha)}$ are continuous. If $X_{11}$ has slowly varying tails, that is $\alpha=0$, then discrete distributions emerge, as the following result shows. 
\begin{theorem}[Poisson behavior for $\alpha=0$]\label{thm:poisson}
Let $\big(N_{-,\ell}\big)_{\ell\ge 1}$ and $\big(N_{+,\ell}\big)_{\ell\ge 1}$ be two i.i.d.\ sequences of independent random variables  with $N_{-,\ell}$ following a Poisson $\mathcal{P}\big(\frac{1}{2\ell}\big)$ distribution. Define 
    \begin{equation*}
        Z_k^{(0)} = \sum_{\ell | k} \ell  (N_{+,\ell } + (-1)^{k/\ell } N_{-,\ell }) \,, \qquad k\ge 1\,.
    \end{equation*}
    For integers $k_1,\ldots, k_\ell\ge 1$, we then have the convergence of traces
        \begin{equation}
        \left(\Tr \left[ \Y^{k_1}_n \right], \dots, \Tr \left[ \Y^{k_\ell}_n \right] \right) \cid \left(Z_{k_1}^{(0)}, \dots, Z_{k_\ell}^{(0)} \right) \,,\qquad \nto\,,
    \end{equation}
    and the following convergence in distribution  of characteristic polynomial $p_n$ for the topology of local uniform convergence in $\D$:
    \begin{equation*}
        p_n \cid F_0 \,, \qquad \nto\,,
    \end{equation*}
    where $F_0$ is the random holomorphic function 
    \begin{equation}
        \label{eq:F_unif_permutations}
        F_0(z) = \exp \left( - \sum_{k \geqslant 1} \frac{z^k}{k} Z_k^{(0)} \right) .
    \end{equation}
\end{theorem}

\begin{remark}  
    The random analytic function $\log(F_0(z))$ 
    is the symmetric version of the logarithm of limiting characteristic polynomial 
    of a uniform permutation matrix \cite{Coste_Lambert_Zhu}, which was given by 
    \begin{equation*}
       -\sum_{k \geqslant 1} \frac{z^k}{k} Z_k'\,,
    \end{equation*}
    where $Z_k'= \sum_{\ell | k} \ell N_\ell $ with $(N_\ell)_{\ell \geqslant 1}$ being a sequence of independent Poisson 
    random variables with respective parameters $(1/\ell)_{\ell \geqslant 1}$. 
\end{remark}

For slowly varying distributions of $X_{11}$, that is $\P(|X_{11}|>x)=L(x)$ with slowly varying function $L$, it is well-known (e.g.~\cite[Theorem~1]{haeusler:mason:1991}) that 
\begin{equation} \label{eq:dhk}
    1-\E\Big[\max_{j=1,\ldots,n} |Y_{1j}|\Big] \to 0\,, \qquad \nto\,.
\end{equation}
This in agreement with the principle of ``one big jump'' which is folklore in extreme value theory. This principle says that one of the heavy-tailed $|X_{1j}|$'s is exceptionally large compared to the rest. The heavier the tails, the more pronounced this effect becomes. The convergence rate in \eqref{eq:dhk}, therefore, depends on how fast the slowly varying function $L$ approaches zero. Slower convergence of $L(x)$ to zero for $x\to \infty$ corresponds to heavier tails and a faster convergence rate in \eqref{eq:dhk}. If $L\to 0$ sufficiently slowly, we thus can obtain
\begin{equation} \label{eq:assumption_L}
    n\Big(1-\E\Big[\max_{j=1,\ldots,n} |Y_{1j}|\Big] \Big) \to 0\,, \qquad \nto\,.
\end{equation}

\textbf{Comparison with one-hot matrices.}
Let us consider the independent rows $\y_i = (Y_{i1}, \ldots, Y_{in})$ for $1 \leqslant i \leqslant n$ and let $k_i = \arg \max_{1 \leq j \leq n} |Y_{ij}|$ be the column index of the largest entry in row $i$. The fact that the $\alpha = 0$ case is structurally similar to permutation matrices can be explained using the following result from \cite{Dong_Heiny_Yao}. For every $i \in [n]$, it holds
\begin{equation}
    \label{eq:conv_to_pm_row}
    \|\y_i - \mathrm{sign}(Y_{i k_i}) \e_{k_i}\|_2 \overset{\P}{\rightarrow} 0 \,, \qquad \nto\,,
\end{equation}
where $\|\cdot\|_2$ denotes the Euclidean norm and $\e_i$ is the $i$-th canonical basis vector of $\R^n$. Let us define the $n\times n$ matrix $\Y_n^{(0)}$ whose $i$-th row is $\mathrm{sign}(Y_{i k_i}) \e_{k_i}$, $i\in [n]$. By definition, $\Y_n^{(0)}$ has independent rows. We call any matrix having the same distribution as $\Y_n^{(0)}$ a \textit{one-hot} matrix. Moreover, one can derive the convergence in Frobenius norm.

\begin{lemma}[Convergence in Frobenius norm]
\label{lem:conv_frobenius}
    Assuming \eqref{eq:assumption_L}, we have the convergence 
    \begin{equation}
        \label{eq:conv_to_pm_frobenius}
        \|\Y_n - \Y_n^{(0)}\|_F \overset{\P}{\rightarrow} 0 \,, \qquad \nto\,.
    \end{equation}
\end{lemma}

In Figure~\ref{fig:prop29-spectral-comparison}, we overlay the spectrum of $\Y_n$ and the spectrum of $\Y_n^{(0)}$. For sufficiently heavy tails (corresponding to  sufficiently small $\beta$), condition \eqref{eq:assumption_L} is satisfied. This partially explains why the distinction between the spectra of $\Y_n$ and $\Y_n^{(0)}$ is more pronounced in the right part of the figure.

Lemma~\ref{lem:conv_frobenius} shows that the matrix $\Y_n$ behaves almost as a one-hot matrix, by having only one nonzero entry equal to $\pm 1$ on each row. We have the following result on one-hot matrices, which is proved in Section \ref{subsec:alpha_0}. Let us consider the spectral measure of the one-hot matrix $\Y_n^{(0)}$ given by
    \begin{equation*}
        \mu_n^{(0)} := \frac{1}{n} \sum_{i=1}^n \delta_{\lambda_i \left(\Y_n^{(0)}\right)} \,,
    \end{equation*}
where $\lambda_1\left(\Y_n^{(0)}\right),\ldots,\lambda_n\left(\Y_n^{(0)}\right)$ are the eigenvalues of $\Y_n^{(0)}$.

\begin{figure}[tb]
    \centering
    \includegraphics[width=0.9\linewidth]{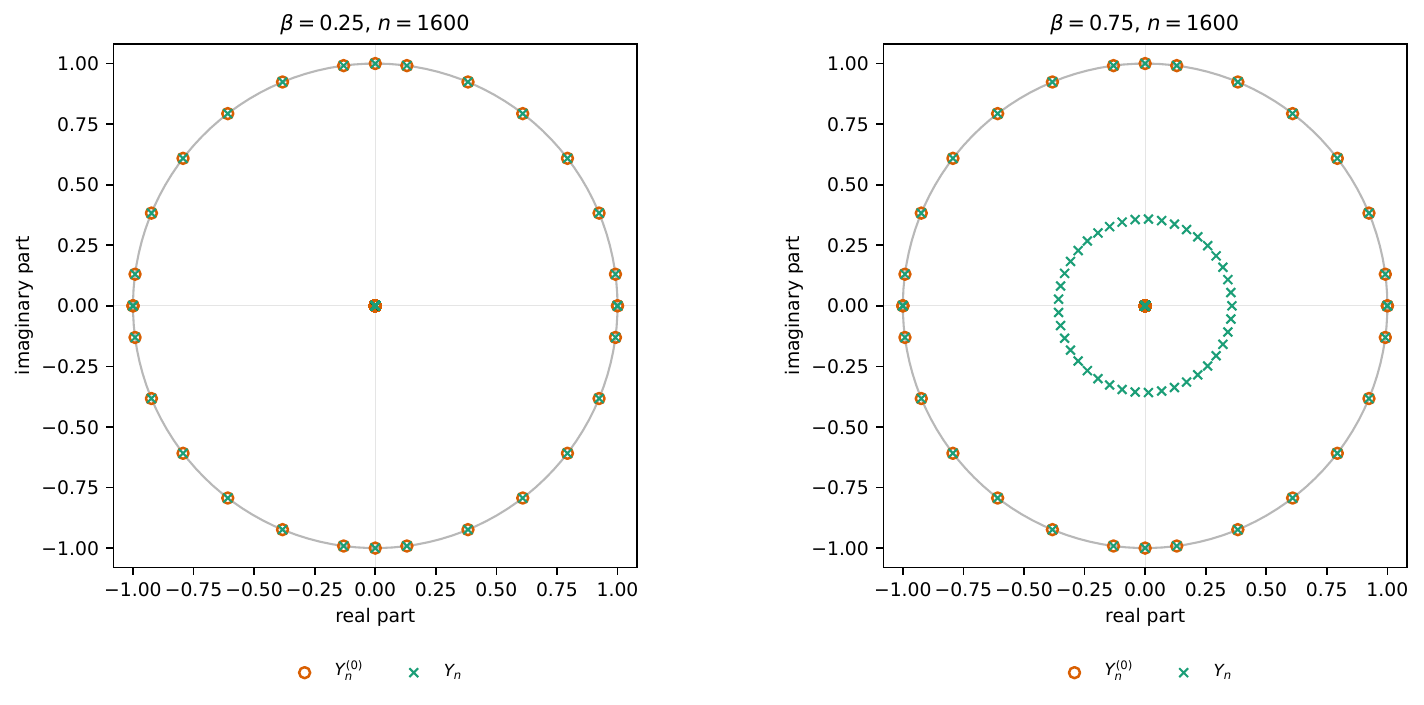}
    \caption{Spectral clouds for two slowly varying models with $n=1600$, where $\mathbb P(|X_{ij}|>x)=(\log x)^{-\beta}$. The left and right panels correspond to $\beta=0.25$ and $\beta=0.75$, respectively. For $\beta=0.25$, $98.06\%$ of the eigenvalues satisfy $|\lambda_j|<0.1$, while for $\beta=0.75$ the corresponding proportion is $95.50\%$.}
    \label{fig:prop29-spectral-comparison}
\end{figure}

\begin{proposition}[Convergence of the spectral measure for one-hot matrices] \label{prop:spectral_measure_one_hot}
   Let $\Y_n^{(0)}$ be a one-hot matrix. Then, its 
        spectral measure converges in probability to the Dirac measure at zero, 
   \begin{equation}
       \label{eq:conv_esd_one_hot_zero}
       \mu_n^{(0)} \cip \delta_0\,, \qquad \nto \,.
   \end{equation}
\end{proposition}

We remark that the nonzero eigenvalues of $\Y_n^{(0)}$ are thus outliers of the limiting spectral distribution which is concentrated at zero. This is somewhat surprising since the singular value measure of $\Y_n$ for $\alpha=0$ remains non-degenerate in the limit. More precisely, from \cite{heiny:yao:2022} it is known that $\mu_{\Y_n\Y_n^\top}$ converges weakly to the Poisson distribution with parameter $1$. The comparison with one-hot matrices and the proof of Proposition~\ref{prop:spectral_measure_one_hot} suggest that the number of nonzero eigenvalues of $\Y_n$ asymptotically behaves like $\sqrt{\frac{n\pi}{2}}$. For those eigenvalues a different normalization will be needed to obtain a non-degenerate limit measure, which is a topic for future research.

\subsubsection{Gaussian behavior} \label{sec:gaussianbehavior}

\begin{figure}[tb]
    \centering
    \includegraphics[width=0.9\linewidth]{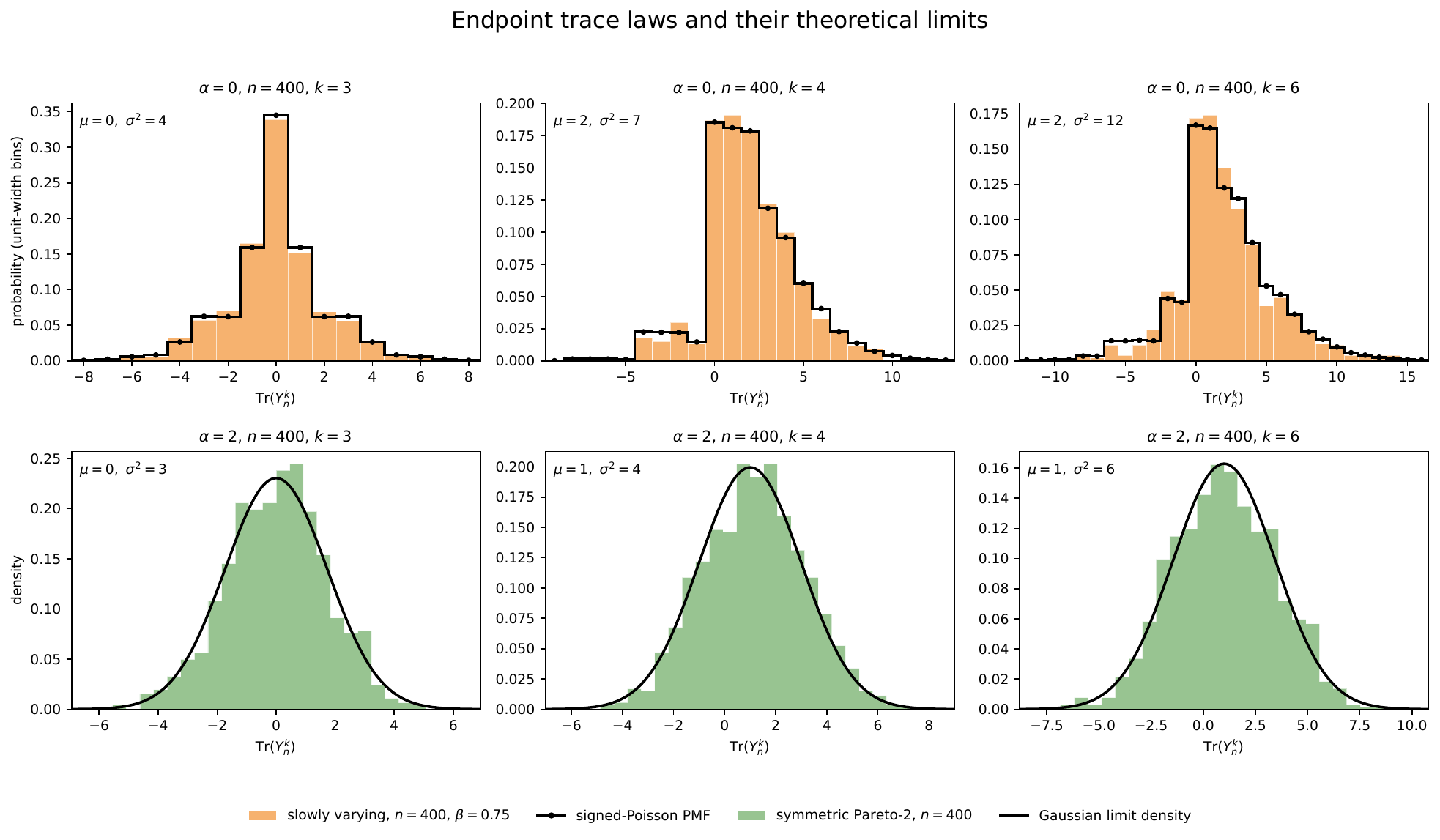}
    \caption{Illustration of Poisson limit in Theorem~\ref{thm:poisson} and the Gaussian limit in Theorem~\ref{thm:gaussian}. Trace distributions at the two endpoints $\alpha=0$ and $\alpha=2$ for $n=400$ and $k\in\{3,4,6\}$ are shown. The upper row compares the slowly varying model $\mathbb P(|X|>x)=(\log x)^{-\beta}$ with the exact signed-Poisson limit. The lower row compares Gaussian and symmetric Pareto$(2)$ entries after standardization by $(\Tr(\Y_n^k)-\1_{\{2\mid k\}})/\sqrt{k}$. The agreement provides a finite-size illustration of the Poisson-to-Gaussian transition.}
    \label{fig:trace-endpoints}
\end{figure}

\begin{theorem}[Phase transition at $\alpha=2$]\label{thm:gaussian}
Assume $X_{11} \in \DAN$. For integers $k_1,\ldots, k_\ell\ge 1$, we then have the convergence of traces
    \begin{equation}
        \left(\Tr \left[ \Y^{k_1}_n \right], \dots, \Tr \left[ \Y^{k_\ell}_n \right] \right) \cid \left(\1_{\{2\mid k_1\}}+\sqrt{k_1}Z_{k_1}',\dots,\1_{\{2\mid k_\ell\}}+\sqrt{k_\ell}Z_{k_\ell}'\right) \,,\qquad \nto\,,
    \end{equation}
    and the limiting characteristic polynomial is  
    \begin{equation}
    \label{eq:limit_gaussian_case}
        p_n(z) \cid  
        \sqrt{1 - z^{2}} \cdot \exp \left( - \sum_{k \geqslant 1} 
        \frac{z^k}{\sqrt{k}} Z_k' \right) \,, \qquad \nto,\, z\in \mathbb{D}\,,
    \end{equation}
    where $(Z_k')_{k \geqslant 1}$ are i.i.d.\ standard Gaussian random variables. 
\end{theorem}

The limiting function in \eqref{eq:limit_gaussian_case} is the real analog of the \textit{Gaussian multiplicative chaos} \cite{Rhodes_Vargas}, which appeared as limit for the characteristic polynomial of i.i.d.\ matrices \cite{Bordenave_Chafai_Garcia} and circular $\beta$-ensembles~\cite{Chhaibi_Najnudel}.

Figure~\ref{fig:trace-endpoints} illustrates the two boundary regimes. At $\alpha=0$, the finite-$n$ slowly varying model is compared with the exact signed-Poisson limit of Theorem~\ref{thm:poisson}. An analogous comparison based on Theorem~\ref{thm:gaussian} is shown for $\alpha=2$.

\subsection{Open questions}

\begin{itemize}
    \item Let us consider the spectral measure of $\Y_n$ given by $\mu_{\Y_n} = \tfrac{1}{n} \sum_{i=1}^n \delta_{\lambda_i(\Y_n)}$. We saw in Proposition~\ref{prop:spectral_measure_one_hot} that for one-hot matrices, one has that 
    \begin{equation*}
        \mu_{\Y_n^{(0)}} \cip \delta_0 \,.
    \end{equation*}
    It is an open problem to extend this convergence to $\mu_{\Y_n}$ for every $\alpha \in [0,2)$.  Simulations suggest that the limit measure would still be the Dirac mass at zero. 
  For $\alpha=0$, using the proximity \eqref{eq:conv_to_pm_row} and 
    \eqref{eq:conv_to_pm_frobenius} between $\Y_n$ and one-hot matrices would be a first step towards the result for $\alpha \in [0,2)$. 
    \item Corollary \ref{cor:upper_bound_spectral_radius} provides an upper bound for the spectral radius. Together with a lower bound, one could achieve convergence of the spectral radius, which we conjecture to converge to one.
 \end{itemize}
 
\section{Proof of Theorem \ref{th:conv_carac_poly}}
\label{sec:proof_main_thm}

The goal of this section is to give the necessary material to prove Theorem \ref{th:conv_carac_poly}. Recall that $\mathcal{H}(\D)$ is the space of holomorphic functions on the unit disk $\D$ endowed with the topology of uniform convergence on compact sets. The proof of Theorem \ref{th:conv_carac_poly} follows from the general result of Lemma~\ref{lem:conditions_a_b} below. For its proof, we refer the reader to Section 4.2 in \cite{Bordenave_Chafai_Garcia}.

\begin{lemma}[Tightness and convergence of coefficients imply convergence of functions]
\label{lem:conditions_a_b}
	Let $\{h_n \}_{n \geqslant 1}$ be a sequence of random elements in $\mathcal{H}(\D)$ and denote the coefficients of $h_n$ by $(\xi_k^{(n)} )_{k \geqslant 0}$ so that for all $z \in \D$, $h_n(z) = \sum_{k \geqslant 0} \xi_k^{(n)} z^k$. Suppose also that the following conditions hold.
	\begin{itemize}
		\item[$(a)$] The sequence $\{h_n \}_{n \geqslant 1}$ is
		a tight sequence of random elements of $\mathcal{H}(\D)$.
		\item[$(b)$] There exists a sequence $(\xi_k)_{k \geqslant 0}$ of random variables such that, for every $m \geqslant 0$, \\
		$\left(\xi_0^{(n)}, \dots, \xi_m^{(n)}\right) \cid (\xi_0, \dots, \xi_m)$.
	\end{itemize}
	Then, $h(z) = \sum_{k \geqslant 0} \xi_k z^k$ is a well-defined function in $\mathcal{H}(\D)$ and $h_n \cid h$ in $\mathcal{H}(\D)$ for the topology of local uniform convergence.
\end{lemma}

\noindent
The first step of the proof of Theorem~\ref{th:conv_carac_poly} is therefore to show the following result.
\begin{proposition}[Tightness]
    \label{thm:tightness}
    The sequence $\{ p_n \}_{n \geqslant 1}$ is tight.
\end{proposition}
\begin{proof}
By Lemma 2.6 of \cite{Shirai}, a sufficient criterion to derive the tightness of a sequence $(p_n)_n$ of holomorphic functions on $\D$ is to bound the function
$z \mapsto \E \left[ |p_n(z)|^2 \right]$ by a deterministic continuous function on any compact subset of $\D$.
Let us write 
\begin{equation*}
    p_n(z) = 1 + \sum_{k=1}^n z^k \Delta_k(\Y_n) \,,
\end{equation*}
where for $k \leqslant n$, 
\begin{equation*}
    \Delta_k(\Y_n) = \sum_{I \subseteq [n] \, : \, |I| = k}
    \det_I(\Y_n) \quad \text {with } \quad \det_I(\Y_n) \coloneqq \sum_{\sigma \in S_I} (-1)^\sigma \prod_{i \in I} Y_{i, \sigma(i)}
\end{equation*}
and where $S_I$ denotes the symmetric group on $I \subseteq [n]$. Using row independence in the matrix $\Y_n$, we have that for $I \neq J$, $\E \left[ \det_I(\Y_n) \det_J(\Y_n) \right] = 0$. Moreover, using the symmetry of the distribution of the $(X_{ij})$ and the identically distributed property, we get that $\E \left[\det_I(\Y_n)^2 \right] = |I|! \prod_{i \in I} \E \left[ Y_{11}^2 \right] = (|I|!) n^{-|I|}$. Tightness therefore follows from
\begin{equation*}
    \E \left[ |p_n(z)|^2 \right] = 1+\sum_{k=0}^n |z|^{2k} \E[(\Delta_k(\Y_n))^2] = \sum_{k=0}^n |z|^{2k} n^{-k} k! \binom{n}{k} 
    \leqslant \sum_{k \geqslant 0} |z|^{2k} = \frac{1}{1 - |z|^2} \,.
\end{equation*}
\end{proof}

The second step of the proof of Theorem~\ref{th:conv_carac_poly} consists in showing the convergence in distribution of the coefficients $\big(\xi_k^{(n)} \big)_{0 \leqslant k \leqslant n}$ of $p_n$. Using that for $z \in \D$ close enough to zero,
\begin{equation*}
    p_n(z) = \exp \left( - \sum_{k \geqslant 1} \frac{z^k}{k} \Tr \left[ \Y_n^k \right] \right) 
\end{equation*}
yields that coefficients $\big(\xi_k^{(n)} \big)_{0 \leqslant k \leqslant n}$ 
are polynomials (which do not depend on $n$) of the traces $\left(\Tr \left[ \Y_n^k \right] \right)_{0 \leqslant k \leqslant n}$. Therefore, showing the convergence in distribution of the coefficients of $p_n$ is equivalent to showing the convergence in distribution of traces. The latter convergence is stated in Theorem~\ref{thm:conv_traces}.

Using Theorem \ref{thm:tightness} and \ref{thm:conv_traces}, we have verified the conditions of Lemma \ref{lem:conditions_a_b} which now yields the desired result of  Theorem~\ref{th:conv_carac_poly} and completes the proof.

\section{Convergence of traces: proof of Theorem \ref{thm:conv_traces}} \label{sec:conv_traces}

The goal of this section is to prove Theorem \ref{thm:conv_traces}. After introducing a general convergence result for moments of entries of $\Y_n$ in Section \ref{subsec:moment_convergence_toolbox}, we first show the convergence of the expectations 
\begin{equation*}
    \left(\E \, \Tr \big[ \Y_n^k \big] \right)_{k \geqslant 1}
\end{equation*}
in Section \ref{subsec:conv_expectation}. In Section \ref{subsec:radnom_contribution}, we then consider the centered random variables 
\begin{equation*}
    \left( \Tr \big[ \Y_n^k \big] - \E \, \Tr \big[ \Y_n^k \big] \right)_{k \geqslant 1}
\end{equation*}
and we show that they converge in distribution using the method of moments (see Proposition \ref{prop:moment_convergence}). Based on the limiting moments in Proposition \ref{prop:moment_convergence}, we finally give the proof of Theorem \ref{thm:conv_traces}.
\medskip

First, we need to introduce some useful notation. We start by writing
\begin{equation*}
    \Tr \left[ \Y_n^k \right] = \sum_{i_1, \dots, i_k=1}^n Y_{i_1i_2} Y_{i_2i_3} \dots 
    Y_{i_ki_1} =: \sum_{\mathbf{i}: [k] \to [n]} Y_\mathbf{i} \,.
\end{equation*}

\noindent
One can think of this sum as a sum over all paths $\mathbf{i}=(i_1, \dots i_k)$ of length $k$ with coefficient $Y_{i_ji_{j+1}}$ associated to the $j$-th edge. Note that from row independence, if $C(\mathbf{i}) \coloneqq \left\{ i_j: j \in [k] \right\}$ has cardinality $k$ then all coefficients in the sum are independent and therefore $\E[Y_\mathbf{i}] = 0$. Such terms will contribute to the random part of the limit. 
\noindent
To any $\mathbf{i} = (i_1, \dots i_k)$, we associate a directed graph $G_\mathbf{i} = (V_\mathbf{i}, E_\mathbf{i})$ with vertex set $V_\mathbf{i} = \{ i_j: j \in [k] \}$ and edge set $E_\mathbf{i} = \bigsqcup_{j \in [k]} (i_{j}, i_{j+1})$. Note that there might be multiple copies of an edge between two vertices in $G_\mathbf{i}$. Finally, for a vertex $v\in V_{\mathbf{i}}$ we define its outer degree $d_v = \sum_{u \in V_{\mathbf{i}}} \1_{(v, u) \in E_\mathbf{i}}\ge 1$.

\subsection{Moments of entries in $\Y_n$}
\label{subsec:moment_convergence_toolbox}

Unless explicitly stated otherwise, the $(X_{it})$ are i.i.d.\ and symmetrically distributed throughout this paper, which implies that the $Y_{it}$ are symmetric as well. The following properties of the matrix $\Y=(Y_{it})$ will be repeatedly used in this paper.
\begin{enumerate}
\item[(1)] By symmetry of the entry distribution we have for $s\le n$ that $\E [Y_{i1}^{m_1}\cdots Y_{is}^{m_s}]=0$ if at least one exponent $m_j\in\N$ is odd.
\item[(2)] $\Y$ has independent rows.
\item[(3)] By definition, $\sum_{t=1}^n Y^2_{it}=1$ for each row $i$.  
\end{enumerate}
Auxiliary results about the mixed moments of $Y_{ij}$'s will be crucial in our analysis.
To this end, define for all positive integers $k_1,\ldots, k_r$ 
\begin{equation*}
    \beta_{2k_1,\ldots, {2k_r}}:=\E \left[Y_{11}^{2k_1}Y_{12}^{2k_2} \cdots Y_{1r}^{2k_r} \right],    
\end{equation*}
where we recall the definition of $Y_{ij}$ from \eqref{def:R}. We introduce a set of positive constants which will be used throughout this paper:
\begin{equation}\label{eq:c_k-def}
    c_1^{(\alpha)}:=1 \quad \text{ and } \quad c_k^{(\alpha)}:=
    \begin{cases}
        \displaystyle
        \frac{\Gamma(k-\alpha/2)}{\Gamma(1-\alpha/2)\Gamma(k)}\,,& \alpha\in [0,2),\\
        \displaystyle
        0\,,& \alpha=2,
    \end{cases} 
\qquad  k\ge 2.
\end{equation}
The following key lemma reveals the asymptotic behavior of $\beta_{2k_1,\ldots, {2k_r}}$.  

\begin{lemma}[Moment asymptotics]
\label{lem:allmoments}
Consider integers $k_1\ge\cdots \ge k_r\ge 1$ and let $\Gamma(\cdot)$ denote the gamma function. 
\begin{itemize}
\item[(a)] For any distribution of $X_{11}$, it holds $\beta_2=1/n$ and
\begin{equation*}
n^r  \beta_{2k_1,\ldots, {2k_r}}=O(1)\,, \qquad \nto\,.
\end{equation*}
If additionally $k_1\ge 2$, then $n^r  \beta_{2k_1,\ldots, {2k_r}}=O(n\beta_4)$.
\item[(b)] If $\alpha\in [0,2)$ and $\P(|X_{11}|>x)=x^{-\alpha} L(x)$ for $x>0$, where $L$ is a slowly varying function at infinity, then it holds
\begin{equation}\label{highestmoment}
\lim_{\nto} n \beta_{2k} = 
\frac{\Gamma(k-\alpha/2)}{\Gamma(1-\alpha/2) \Gamma(k)} = c_k^{(\alpha)} \,, \qquad k\ge 2\,.
\end{equation}
\item[(c)] If $X_{11}\in \DAN$, then
\begin{equation}\label{eq:highestmoment}
\lim_{\nto} n \beta_{2k} = 
0 = c_k^{(2)} \,, \qquad k\ge 2\,.
\end{equation}
\end{itemize}
\end{lemma}
\begin{proof}
Since $Y_{11}^2+\cdots+Y_{1n}^2=1$, an application of the multinomial theorem shows that for $k=k_1+\cdots+k_r$,
\begin{equation}\label{eq:dge}
1=\E\left[ \left(Y_{11}^2+\cdots+Y_{1n}^2 \right)^k\right]=
\sum_{r=1}^{k}
 \mathop{\sum_{m_1+\cdots +m_r=k }}_{m_j \ge 1} \binom{n}{r}  \binom{k}{m_1,\ldots, m_r} 
\beta_{2m_1,\ldots, {2m_r}}\,.
\end{equation}
Because all terms on the \rhs\ are nonnegative, we obtain $1\gtrsim n^r \beta_{2k_1,\ldots, {2k_r}}$. Moreover, if $k_1\ge 2$, then we have
\begin{equation*}
\beta_4 \ge \beta_{2k_1}=\E\left[Y_{11}^{2k_1} (Y_{11}^2+\cdots+Y_{1n}^2)^{k-k_1} \right]
\gtrsim n^{r-1} \beta_{2k_1,\ldots, {2k_r}}\,,
\end{equation*}
completing the proof of part (a). 

For $\alpha\in (0,2)$, the statement in part (b) is equation (3.3) in \cite{heiny:yao:2022}. Thus, we only need to consider $\alpha=0$ and show $n\beta_{2k}\to 1, k\ge 2$.  Proposition 3 in \cite{mason:zinn:2005} asserts that $X_{11}$ is slowly varying if and only if $n\beta_4 \to 1$. The latter implies $\beta_{2,2}=o(n^{-2})$. We get for any $r\ge 2$,
\begin{equation*}
\beta_{2,2}=\E\big[Y_{11}^{2} Y_{12}^{2} (Y_{11}^2+\cdots+Y_{1n}^2)^{k-2} \big]\gtrsim n^{r-2} \beta_{2k_1,\ldots, {2k_r}}\,,
\end{equation*}
which means that $n^{r} \beta_{2k_1,\ldots, {2k_r}}=o(1)$ whenever $r\ge 2$. In conjunction with \eqref{eq:dge}, we deduce that $n\beta_{2k} \to 1$ for any $k\ge 2$.

Finally, regarding part (c), we note that $\beta_{2k}\le \beta_4$ for $k\ge 2$, so that it suffices to prove $n \beta_4 \to 0$. According to Theorem 2.4 in \cite{doernemann:heiny:2025}, $n \beta_4 \to 0$ is equivalent to $X_{11}\in \DAN$, which concludes the proof of the lemma.
\end{proof}

\begin{remark}
It is worth noting that regularly varying distributions with index $\alpha\ge 2$ are always in the domain of attraction of the normal distribution. For such distributions we see that the limit in \eqref{eq:highestmoment} does not depend on the specific value $\alpha$.   
\end{remark}

\subsection{Mean convergence}
\label{subsec:conv_expectation}

Let us define the constants
    \begin{equation}
        \label{eq:def_tilde_ck}
        \tilde{c}_{2k}^{(\alpha)} \coloneqq \sum_{\ell | k} 
        \left(c_{\ell}^{(\alpha)}\right)^{k / \ell} 
        \text{ and } \ \tilde{c}_{2k+1}^{(\alpha)} \coloneqq 0 \,, \qquad k\ge 1\,.
    \end{equation}
\begin{lemma}[Mean convergence] \label{lemma:mean_convergence}
For every $k \geqslant 1$, one has the convergence
    \begin{equation}
        \E \, \Tr \left[ \Y^k_n \right] \to \tilde{c}_k^{(\alpha)} \,, \qquad n \to \infty\,.
    \end{equation}
\end{lemma}

\begin{proof}[Proof of Lemma \ref{lemma:mean_convergence}]
    Let $k \geqslant 1$ be fixed and set $b_{k, n} = \Tr \left[ \Y^k_n \right]$.
    From the symmetry and row independence, one has 
    $\E [b_{2k+1, n}] = 0$. For the rest of the proof, we will only consider even moments $\E [b_{2k, n}]$. Let $G_\mathbf{i}$ be the directed graph associated to path $\mathbf{i}$. Since $\mathbf{i}$ is a cycle, the graph $G_\mathbf{i}$ is connected. Then, every edge should have an even multiplicity for $\mathbf{i}$ to have a nonzero contribution to $\E [b_{2k, n}]$. The number of vertices $|V_\mathbf{i}|$  is equal to the number of different rows appearing in the product $Y_\mathbf{i}$. Using Lemma~\ref{lem:allmoments}, each vertex $v \in V_\mathbf{i}$ will give an asymptotic contribution of order $n^{-d_v}$, where $d_v$ is the outer degree of $v$. Therefore, to have a non-vanishing contribution, each vertex should be of outer degree one; that is, each row should have only one column appearing in a product expression of the form $\beta_{2k_1, \dots, 2k_r}$. The graph $G_\mathbf{i}$ is thus a cycle.
    Let $b \coloneqq \# \left\{ 2 \leqslant j \leqslant 2k+1 \mid i_j = i_1 \right\} \geqslant 1$ be the number of returns to the first vertex of $\mathbf{i}$, with the convention that $i_{2k+1} = i_1$ and let $2 \leqslant j_1 < \dots < j_b \leqslant 2k+1$ be the indexes such that $i_{j_r} = i_1$. 
    Since each vertex has outer degree one, for each $1 \leqslant r \leqslant b$, we have that $i_{j_r+1} = i_{2}$. Since otherwise, if $i_{j_r+1} \neq i_{2}$, the vertex $i_1$ would have outer degree at least two since $G_\mathbf{i}$ would have edges $(i_1, i_{2})$ and $(i_1, i_{j_r+1})$ having different endpoints. The same holds for every other vertex than $i_1$ using cyclic permutation. Therefore, each edge in the cycle $G_\mathbf{i}$ has the same multiplicity $b \geqslant 1$ which divides $2k$.
    
    Let us thus introduce $d_2(2k) = \{ \ell \in 2 \N : \ell | 2k \}$ the set of even divisors of $2k$. For each $\ell \in d_2(2k)$, consider the cycle on $q = 2k / \ell$ vertices having edges of multiplicity $\ell$ between two consecutive vertices. Such a cycle gives a contribution of $\beta_{\ell}^q$.
    Since there are $n(n-1) \dots (n-q+1) \sim n^q$ such cycles, we thus have 
    \begin{equation*}
        \E [b_{2k, n}] = \sum_{\ell \in d_2(2k)} 
        \left(c_{\ell/2}^{(\alpha)}\right)^{2k / \ell} + o(1) =\tilde{c}_{2k}^{(\alpha)}+o(1)\, .
    \end{equation*}
\end{proof}

\subsection{Random contribution}
\label{subsec:radnom_contribution}

We will prove Theorem \ref{thm:conv_traces} using the method of moments.
To this end, for $k \geqslant 1$, we need the notation $a_{k, n} := \Tr \left[ \Y_n^k \right] 
- \E \, \Tr \left[ \Y_n^k \right]$.

\begin{proposition}[Convergence of moments]
\label{prop:moment_convergence}
    Let $\ell \geqslant 1$ be fixed and $k_1,\ldots,k_\ell\ge 1$. Then it holds
\begin{equation}
    \label{eq:random_limit}
    \E \left[ \prod_{r=1}^\ell a_{k_r, n} \right]
        \to \E \left[ \prod_{r=1}^\ell \tilde{Z}_{k_r} \right] 
        \,, \qquad \nto\,,
\end{equation}
    where $\tilde{Z}_k^{(\alpha)} := Z_k^{(\alpha)} - \E \left[Z_k^{(\alpha)} \right]$ and $\big(Z_k^{(\alpha)}\big)_{k \geqslant 1}$ is the $\alpha$-family from Definition \ref{def:alpha_family}.
\end{proposition}

Theorem \ref{thm:conv_traces} is then a  consequence of Proposition~\ref{prop:moment_convergence}, as we show next.

\begin{proof}[Proof of Theorem \ref{thm:conv_traces}]
From \eqref{eq:def_V_q_m} recall the random variables $\big(V_{q, m}^{(\alpha)}\big)_{q, m \geqslant 1}$. For simplicity we omit the superscript $(\alpha)$ in what follows. Since the $(V_{q, m})_{q, m \geqslant 1} $ are determined by their moments, 
so are the variables $(\tilde{Z}_k^{(\alpha)})_{k \geqslant 1}$ and therefore, using Proposition~\ref{prop:moment_convergence} and  
Lemma \ref{lemma:mean_convergence}, 
\begin{equation}
    \left(\Tr \left[ \Y_n^{k_1} \right], \dots, \Tr \left[ \Y_n^{k_\ell} \right] \right) \cid \left(\tilde{Z}_{k_1}^{(\alpha)}, \dots, \tilde{Z}_{k_\ell}^{(\alpha)}\right) + \left(\tilde{c}_{k_1}^{(\alpha)}, \dots, \tilde{c}_{k_\ell}^{(\alpha)} \right) \,, \qquad \nto\,,
\end{equation}
where $\tilde{c}_{k}^{(\alpha)}$ are defined in \eqref{eq:def_tilde_ck}.
Moreover, writing $\kappa_s(\cdot)$ for the $s$-th cumulant, we have
\begin{equation}
    \label{eq:E_Z_k_as_tilde_c}
    \E \left[Z_{k}^{(\alpha)}\right] = 
    \sum_{\ell | k} \E \left[V_{\ell, \frac{k}{\ell}} \right] = 
    \sum_{\ell | k} \kappa_1 \left(V_{\ell, \frac{k}{\ell}} \right)
    = \sum_{\ell | k} \left(c_{\frac{k}{2\ell}}^{(\alpha)}\right)^{\ell} \1_{\frac{k}{\ell} \in 2 \N} = \tilde{c}_k^{(\alpha)}
    \,.
\end{equation}
For the last equality we used the fact that $\left\{(k/(2\ell),\ell): \ell | \tfrac{k}{2} \right\} = \left\{(\ell,k/(2\ell)): \ell | \tfrac{k}{2} \right\}$.
Therefore, we conclude that 
\begin{equation}
    \left(\Tr \left[ \Y^{k_1}_n \right], \dots, \Tr \left[ \Y^{k_\ell}_n \right] \right) \cid \left( Z_{k_1}^{(\alpha)}, \dots, Z_{k_\ell}^{(\alpha)} \right) \,, \qquad \nto\,.
\end{equation}
\end{proof}

\begin{remark}[Product expression] 
 Since for every $\ell \geqslant 1$, we have that $\big|c_\ell^{(\alpha)} \big| \leqslant 1$, one has 
$\big|\tilde{c}_{2k}^{(\alpha)} \big| \leqslant 2k$ so that the series
\begin{equation}
    g(z) := -\sum_{k \geqslant 1} \frac{z^k}{k} \tilde{c}_{k}^{(\alpha)} =-\sum_{k \geqslant 1} \frac{z^{2k}}{2k} \tilde{c}_{2k}^{(\alpha)}= 
    - \sum_{k \geqslant 1} \frac{z^{2k}}{2k} \sum_{\ell \in d_2(2k)} \left(c_{\ell/2}^{(\alpha)}\right)^{2k / \ell}
\end{equation}
converges absolutely on compact subsets of $\D$. Recall that since $\E \left[Z_{k}^{(\alpha)}\right] = \tilde{c}_{k}^{(\alpha)}$, we have that $g(z) = - \sum_{k \geqslant 1} \frac{z^k}{k} \E \left[Z_{k}^{(\alpha)}\right]$. Writing $2k = (2\ell)q$ and using Fubini yields
\begin{equation*}
    g(z) = -\sum_{\ell \geqslant 1} \frac{1}{2 \ell} \sum_{q \geqslant 1} \frac{ \left(z^{2 \ell} c_\ell^{(\alpha)} \right)^q}{q} 
    = \sum_{\ell \geqslant 1} \frac{1}{2 \ell} \log \left( 1 - z^{2 \ell}c_\ell^{(\alpha)} \right) 
    = \log \left( \prod_{\ell \geqslant 1} \left( 1 - z^{2 \ell}c_\ell^{(\alpha)} \right)^{\frac{1}{2 \ell}} \right) \,.
\end{equation*}
Therefore, using that $Z_{k}^{(\alpha)} = \tilde{c}_{k}^{(\alpha)} +  \tilde{Z}_{k}^{(\alpha)}$, the limit expression of $p_n$ can be written as
\begin{equation}
\label{eq:product_expression}
    p_n(z) \cid \prod_{\ell \geqslant 1} 
    \left( 1 - z^{2 \ell}c_\ell^{(\alpha)} \right)^{\frac{1}{2 \ell}}  \cdot \exp \left( - \sum_{k \geqslant 1} \tfrac{z^k}{k} \tilde{Z}_{k_\ell}^{(\alpha)} \right) \,, \qquad \nto\,.
\end{equation}
\end{remark}

The product expression \eqref{eq:product_expression} above can be used to identify the limiting zeroes of $p_n$. Indeed, we have the following result from Proposition 2.3 in \cite{Shirai}, see also Proposition 7.1 in \cite{Coste_Bernoulli}. 

\begin{lemma}[Convergence of zeroes]
\label{lem:zeroes_convergence}
    Let $(f_n)_{n \geqslant 1}$ be a sequence of random analytic functions in $\D$, converging in law to a non-zero function $f$. Let $\Phi_n$ and $\Phi$ be the respective sets of zeroes of $f_n$ and $f$ respectively. Then, $\Phi_n \cid \Phi$, for the topology of weak convergence with respect to the vague topology.
\end{lemma}

The rest of this section is devoted to the proof of Proposition \ref{prop:moment_convergence}. Section \ref{subsec:moment_expansion} establishes a first formula for the limit of centered traces. The proof of Proposition \ref{prop:moment_convergence} will be given in Section~\ref{subsec:proof_prop_random_limit}.

\subsubsection{Moment expansion}
\label{subsec:moment_expansion}

Let $\mathcal{P}_\ell$ denote the set of set partitions $\pi$  of $\{1,\ldots,\ell\}$.  Furthermore, we write for $\mathcal{P}_{\ell,2}$ the set of set partitions of $\{1,\ldots,\ell\}$ having blocks $B$ of size at least two. In other words, $\pi\in \mathcal{P}_{\ell,2}$ if and only if $\pi\in \mathcal{P}_{\ell}$ and $|B|\ge 2$ for all $B\in\pi$.

\noindent
The following result is an important ingredient of the proof of Proposition \ref{prop:moment_convergence}.

\begin{lemma}[Product formula]
\label{lem:product_formula}
Let $\ell \geqslant 1$ be fixed and $k_1,\ldots,k_\ell\ge 1$. With the convention $c_x^{(\alpha)}=0,\, x\notin\mathbb N$, then it holds 
    \begin{equation}
    \label{eq:moments_and_cgd}
        \E \left[ \prod_{r=1}^\ell a_{k_r, n} \right]
        \to \sum_{\pi \in \mathcal{P}_{\ell,2}} \prod_{B \in \pi} \left(  
     \sum_{q | \gcd(k_j:j\in B)} 
         q^{|B|-1}\left(c_{k_B/(2q)}^{(\alpha)} \right)^q \right) \,, \qquad \nto\,,
    \end{equation}
    where we used the notation $k_B = \sum_{j \in B} k_j$ for a given block $B$ of the partition $\pi$. 
\end{lemma}
\begin{proof}[Proof of Lemma \ref{lem:product_formula}]
    For a path $\mathbf{i}: [k] \to [n]$, let us denote by $\overset{\circ}{Y_\mathbf{i}} = Y_\mathbf{i} - \E Y_\mathbf{i}$. Then we have that 
    \begin{equation}
    \label{eq:intial_sum}
        \E \left[ \prod_{r=1}^\ell a_{k_r, n} \right] = 
        \E \left[ \prod_{r=1}^\ell \left( \sum_{\mathbf{i}^{(r)}: [k_r] \to [n] } \overset{\circ}{Y_{\mathbf{i}^{(r)}}} \right) \right] = 
        \sum_{\mathbf{i}^{(1)}, \dots, \mathbf{i}^{(\ell)}} \E \left[ \prod_{r=1}^\ell \overset{\circ}{Y_{\mathbf{i}^{(r)}}} \right] \ .
    \end{equation}
    Let us start by proving that
    \begin{equation}
        \label{eq:reduced_to_cycles}
        \sum_{\mathbf{i}^{(1)}, \dots, \mathbf{i}^{(\ell)}} \E \left[ \prod_{r=1}^\ell \overset{\circ}{Y_{\mathbf{i}^{(r)}}} \right] 
        = \sum_{\substack{
\mathbf{i}^{(1)}, \dots, \mathbf{i}^{(\ell)} \\
(\mathbf{i}^{(1)}, \dots, \mathbf{i}^{(\ell)}) \in \mathcal{C}_{k_1} \times \dots \times \mathcal{C}_{k_\ell}
}} \E \left[ \prod_{r=1}^\ell \overset{\circ}{Y_{\mathbf{i}^{(r)}}} \right] + o(1) \,,
    \end{equation} 
    where $\mathcal{C}_{k}$ denotes the set of directed loop paths of length $k$ having every vertex of outer degree one. Recall from Lemma \ref{lem:allmoments} and row independence that for a path $\mathbf{i}$
    \begin{equation}
        \label{eq:control_expectation}
        \E [Y_\mathbf{i}] = O \left(n^{- \sum_{v \in V_{\mathbf{i}}} d_v} \right)\,,
    \end{equation}
    where $V_{\mathbf{i}}$ is the vertex set of the graph $G_{\mathbf{i}}$, and  
    $d_v \geqslant 1$ is the outer degree of $v$ defined in the proof of Lemma \ref{lemma:mean_convergence}. 
    A generic term in the sum \eqref{eq:intial_sum} is of the form 
\begin{equation*}
                \E \left[Y_{\mathbf{i}^{(1)}}^{s_1} 
        \dots Y_{\mathbf{i}^{(\ell)}}^{s_\ell} \right]  = \prod_{r: s_r = * } 
        \E[Y_{\mathbf{i}^{(r)}}] \cdot \E \left[ \prod_{r: s_r = \cdot} Y_{\mathbf{i}^{(r)}}  \right]\,,
    \end{equation*}
    where $s_i \in \{*, \cdot\}$ with $x^\cdot = x$ and $x^* = \E[x]$.
    Using \eqref{eq:control_expectation}, the first term in the product on the right-hand side is 
    \begin{equation*}
        \prod_{r: s_r = * } 
        \E[Y_{\mathbf{i}^{(r)}}] = O \left(n^{- \sum_{r: s_r = *} \sum_{v \in V_{\mathbf{i}^{(r)}}} d_v^{(r)}} \right)\,.
    \end{equation*}
    Let us define the graphs $G = (V, E)$ and $\tilde{G} = (\tilde{V}, \tilde{E})$ by 
    \begin{align*}
        V = \bigcup_{r=1}^\ell V_{\mathbf{i}^{(r)}}\,, \quad E = \bigsqcup_{r=1}^\ell E_{\mathbf{i}^{(r)}}\,, \quad &\text{ and } \quad
        \tilde{V} = \bigcup_{r: s_r = \cdot}^\ell V_{\mathbf{i}^{(r)}} \,, \quad\tilde{E} = \bigsqcup_{r: s_r= \cdot}^\ell E_{\mathbf{i}^{(r)}}\,.
    \end{align*}
    Then, the second term has contribution 
    \begin{equation*}
        \E \left[ \prod_{r: s_r = \cdot} Y_{\mathbf{i}^{(r)}}  \right] = O \left(n^{- \sum_{v \in \tilde{V}} \tilde{d}_v} \right),
    \end{equation*}
    where $\tilde{d}_v$ is the outer degree of a vertex $v$ of the graph $\tilde{G}$. 
    Up to a relabeling of the vertices, the number of $\mathbf{i}^{(1)}, \dots, \mathbf{i}^{(\ell)}$ forming the graph $G$ is of order $O \left(n^{|V|} \right)$.  Thus, for the contribution to be non-vanishing one must have 
    \begin{equation}
        |V| \geqslant \sum_{v \in \tilde{V}} \tilde{d}_v + \sum_{r: s_r = *} \sum_{v \in V_{\mathbf{i}^{(r)}}} d_v^{(r)} \,.
    \end{equation}
    However, $|V| \leqslant |\tilde{V}| +  \sum_{r: s_r = *} |V_{\mathbf{i}^{(r)}}| \leqslant \sum_{v \in \tilde{V}} \tilde{d}_v + \sum_{r: s_r = *} \sum_{v \in V_{\mathbf{i}^{(r)}}} d_v^{(r)}$.
    Therefore, the non-vanishing contribution are the ones coming from paths $\mathbf{i}^{(1)}, \dots, \mathbf{i}^{(\ell)}$ 
    such that 
    \begin{equation}
        \label{eq:non_vanishing_1}
        |V| = \sum_{v \in \tilde{V}} \tilde{d}_v + \sum_{r: s_r = *} \sum_{v \in V_{\mathbf{i}^{(r)}}} d_v^{(r)} \,,
    \end{equation}
    that is, for which $ \tilde{d}_v = 1$ for $v \in \tilde{V}$ and $d_v^{(r)} = 1$ for $r$ such that $s_r = *$. 
    Moreover, the equality implies that $|V| = |\tilde{V}| +  \sum_{r: s_r = *} |V_{\mathbf{i}^{(r)}}| $ so that $\left( \bigcup_{r: s_r = *} V_{\mathbf{i}^{(r)}} \right) \ \cap \tilde{V} = \emptyset$. Therefore, for every $v \in V$, $d_v = 1$ so that non-vanishing contributions are given by tuples $(\mathbf{i}^{(1)}, \dots, \mathbf{i}^{(\ell)})  \in \mathcal{C}_{k_1} \times \dots \times \mathcal{C}_{k_\ell}$ which proves \eqref{eq:reduced_to_cycles}. 
    \medskip
    
    Since each term in the product $\E \left[ \prod_{r=1}^\ell \overset{\circ}{Y_\mathbf{i^{(r)}}} \right] $ is centered, non-vanishing terms come from partitions of $[\ell]$ such that on each block $B = (b_1, \dots, b_s)$ of size $s \geqslant 2$, 
    and $G_{\mathbf{i}^{(b_1)}}, \dots, G_{\mathbf{i}^{(b_s)}}$ are cycles with common vertex set having edges in the same direction. Thus, 
    \begin{equation}
    \sum_{\substack{\mathbf{i}^{(1)}, \dots, \mathbf{i}^{(\ell)} \\
    (\mathbf{i}^{(1)}, \dots, \mathbf{i}^{(\ell)}) \in \mathcal{C}_{k_1} \times \dots \times \mathcal{C}_{k_\ell}}} 
    \E \left[ \prod_{r=1}^\ell \overset{\circ}{Y_{\mathbf{i}^{(r)}}} \right]
    = \sum_{\pi \in \mathcal{P}_{\ell,2}} \prod_{B \in \pi} \left( \sum_{\mathbf{i}^{(b_1)}, \dots, \mathbf{i}^{(b_s)}} \E \left[ \prod_{r=1}^s \overset{\circ}{Y_{\mathbf{i}^{(b_r)}}} \right] \right) \,.
    \end{equation}
    Let $B$ be a block of $\pi$ of size $s \geqslant 2$.  Up to relabeling, let us assume that $B = (1, \dots, s)$. 
    Recall that $G_{\mathbf{i}^{(b_1)}}, \dots, G_{\mathbf{i}^{(b_s)}}$ must be oriented cycles on a same vertex set. 
    Let $q \geqslant 1$ be the number of vertices in the cycle. 
    Note that for such tuples, as $s \geqslant 2$,
    \begin{equation}
        \E \left[ \prod_{r=1}^s \overset{\circ}{Y_{\mathbf{i}^{(b_r)}}} \right] 
        = \E \left[ \prod_{r=1}^s Y_{\mathbf{i}^{(b_r)}} \right] (1+ o(1))
    \end{equation}
    since any term having at least one $s_r = *$ will give an additional factor of order $O(n^{-q})$. 
    Then, $q | k_1, \dots, q | k_s$ so that there exist integers $\ell_1, \dots, \ell_s$ such that $k_j = q \ell_j$. Between two consecutive vertices of the cycle, there are $\sum_{j=1}^s \ell_j$ edges, which must be even for the expectation to be nonzero. Therefore, $k_B = \sum_{j\in B} k_j$ must be even. 
    For such a $q | \gcd(k_1, \dots, k_s)$, using Lemma~\ref{lem:allmoments} we have 
    \begin{equation}
        \sum_{\mathbf{i}^{(b_1)}, \dots, \mathbf{i}^{(b_s)}}  \E \left[ \prod_{r=1}^s Y_{\mathbf{i}^{(b_r)}} \right] = \sum_{q | \gcd(k_1, \dots, k_s)} 
        q^{s-1}\left(c_{k_B/(2q)}^{(\alpha)} \right)^q + o(1) \,,
    \end{equation}
    where the factor $q^{s-1}$ comes from the $q$ choices of the origin vertex for each of the paths $\mathbf{i}^{(b_2)}, \dots, \mathbf{i}^{(b_s)}$ once $\mathbf{i}^{(b_1)} $ is chosen. Therefore, 
    \begin{equation}
    \label{eq:to_identify}
        \E \left[ \prod_{r=1}^\ell a_{k_r, n} \right]
        \to \sum_{\pi \in \mathcal{P}_{\ell,2}} \prod_{B \in \pi} \left(  
     \sum_{q | \gcd(k_1, \dots, k_s)} 
         q^{s-1}\left(c_{k_B/(2q)}^{(\alpha)} \right)^q \right)  
         \,, \qquad \nto\,,
    \end{equation}
    which completes the proof. 
\end{proof}

\begin{example}{\em
    Let us consider $\ell = 4$ and $(k_1, k_2, k_3, k_4) = (2, 2, 2, 2)$. Then, $\gcd(k_1, \dots, k_s) = 2$ so that the innermost sum is over $q \in \{1, 2\}$. The partitions of $\{1, 2, 3, 4 \}$ contributing to \eqref{eq:to_identify} are 
    \begin{itemize}
        \item $\pi = \{1, 2, 3, 4 \}$ having only one block. Writing $c_k$ instead of $c_k^{(\alpha)}$, we get a contribution of $c_4 + 2^3c_2^2$ for $q=1$ and $q=2$ with $s=4$ respectively.
        \item Three partitions of the form $\pi = \{x, y \} \cup \left(\{1, 2, 3, 4 \} \setminus \{ x, y\}\right)$ giving the same contribution. For each one, the product over blocks is the block of size two squared as there are only two blocks of size two. A block of size two gives a sum over $q \in \{1, 2\}$ and $k_B = 4$ so that the total contribution is $3(c_2 + 2c_1^2)^2$.
    \end{itemize}
    Thus the limit is 
    \begin{equation*}
        c_4 + 11c_2^2 + 12c_1^4 + 12c_2c_1^2 \,.
    \end{equation*}
    }
\end{example}

\subsubsection{Proving Proposition \ref{prop:moment_convergence}}
\label{subsec:proof_prop_random_limit}

\noindent
Before turning to the proof of Proposition \ref{prop:moment_convergence}, let us recall two lemmas on cumulants of random variables. 
\begin{lemma}[Moment-cumulant formula {\cite[Proposition~3.2.1]{PeccatiTaqqu2011}}]
    Let $(Z_k)_{k \geqslant 1}$ be a family of random variables. Then it holds for any positive integers $k_1,\ldots,k_{\ell}$ that
    \begin{equation}
        \label{eq:moment_cumulant_formula}
        \E \left[\prod_{r=1}^\ell Z_{k_r} \right] = \sum_{\pi \in \mathcal{P}_\ell} \prod_{B \in \pi} \kappa_B(Z_{k_1}, \dots, Z_{k_\ell}) \,,
    \end{equation}
    where $\mathcal{P}_{\ell}$ denotes the set of all partitions of $\{1,\ldots,\ell\}$ and where for $B=\{r_1,\ldots,r_s\}\subseteq\{1,\ldots,\ell\}$,
    \begin{equation*}
        \kappa_B(Z_{k_1}, \dots, Z_{k_\ell}) := \kappa_s(Z_{k_{r_1}}, \dots, Z_{k_{r_s}})
    \end{equation*}
    denotes the $s$-th joint cumulant of 
    $Z_{k_{r_1}}, \dots, Z_{k_{r_s}}$.
\end{lemma}

\noindent
Beware that \eqref{eq:moment_cumulant_formula} features partitions with blocks of size one, which do not appear in the limit expression \eqref{eq:to_identify}. The centering 
$\tilde{Z} = Z - \E[Z]$ precisely removes these blocks:
\begin{equation*}
    \kappa_s \left(\tilde{Z}_{k_1}, \dots, \tilde{Z}_{k_s} \right) = 
    \begin{cases}
         \kappa_s(Z_{k_1}, \dots, Z_{k_s}) &\text{ if } s \geqslant 2 \\
         0 &\text{ if } s = 1.
    \end{cases}
\end{equation*}
Therefore, we have 
    \begin{equation}
        \label{eq:moment_cumulant_formula2}
        \E \left[\prod_{r=1}^\ell \tilde{Z}_{k_r} \right] = \sum_{\pi \in \mathcal{P}_\ell} \prod_{B \in \pi} \kappa_B \left(\tilde{Z}_{k_1}, \dots, \tilde{Z}_{k_\ell} \right)  = \sum_{\pi \in \mathcal{P}_{\ell,2}} \prod_{B \in \pi} \kappa_B(Z_{k_1}, \dots, Z_{k_\ell}) \,.
    \end{equation}

\begin{lemma}[Cumulants for division family]
\label{lem:cumulants_division}
Let $(V_{\ell, m})_{\ell, m \geqslant 1}$ be a family of random variables such that the families $\left( V_{\ell, \cdot} \right)_\ell$ 
are independent.
Let us define the family $(Z_k)_{k \geqslant 1}$ by $Z_k = \sum_{\ell | k} V_{\ell, k/\ell}$. Then one has
    \begin{equation}
        \label{eq:cumulants_Z}
        \kappa_s(Z_{k_1}, \dots, Z_{k_s}) = \sum_{q | \gcd(k_1, \dots, k_s)} \kappa_s(V_{q, k_1/q}, \dots, V_{q, k_s/q}) \,.
    \end{equation}
\end{lemma}

\begin{proof}
    The formula \eqref{eq:cumulants_Z} is obtained by the multilinearity of the function $\kappa_s$ together with independence relations.
\end{proof}

\begin{proof}[Proof of Proposition \ref{prop:moment_convergence}]
Let $\big(Z_k^{(\alpha)}\big)_{k \geqslant 1}$ be an $\alpha$-family as in Definition \ref{def:alpha_family}. In particular, recall from \eqref{eq:def_V_q_m} that $V_{q,m}^{(\alpha)}$ is a integral with respect to some Poisson process. We will use the fact that for a Poisson process $N$ with intensity $\nu$, we have the general cumulant formula 
\begin{equation}
    \label{eq:cumulant_Poisson_process}
    \kappa_s \left( \int f_1(\xi) N(\dint \xi), \dots, \int f_s(\xi) N(\dint \xi) \right) 
    = \int f_1(x) \dots f_s(x) \nu(\dint x) \,,
\end{equation}
which comes from the following Laplace functional expression, see (1.2) in \cite{Poisson_book},
\begin{equation}
    \label{eq:laplace_poisson}
    \E \left[ \exp \left( - \int f(\xi) N(\dint \xi) \right) \right] 
    = \exp \left( - \int \left(1 - e^{-f(x)} \right) \nu(\dint x) \right) \,.
\end{equation}
Applying \eqref{eq:cumulant_Poisson_process} with $m_1, \dots, m_s \geqslant 1$ and writing $m = m_1 + \cdots + m_s$, we obtain 
\begin{align}
    \kappa_s \left(V_{q, m_1}^{(\alpha)}, \dots, V_{q, m_s}^{(\alpha)} \right) &= q^s \int f_{m_1}(x) \cdots f_{m_s}(x) \nu_q^{(\alpha)}(\dint x) = 
    q^s \int f_{m}(x)  \nu_q^{(\alpha)}(\dint x)\\
    &= \begin{cases}
        q^{s-1} \int x^{m/2-1} \mu_q^{(\alpha)}(\dint x) = q^{s-1} \E \left[W_q^{m/2-1} \right] = q^{s-1} \big(c_{m/2}^{(\alpha)}\big)^q 
        &\text{ if } m \text{ is even} \\
         0 &\text{ if } m \text{ is odd.} \\
    \end{cases}  
\end{align}
Using Lemma \ref{lem:cumulants_division} and \eqref{eq:moment_cumulant_formula}, one can write the \rhs\ of \eqref{eq:moments_and_cgd} as
\begin{equation}
    \label{eq:identification}
    \sum_{\pi \in \mathcal{P}_{\ell,2}} \prod_{B \in \pi} \left(  
     \sum_{q | \gcd(k_j:j\in B)} 
         q^{|B|-1}\left(c_{k_B/(2q)}^{(\alpha)} \right)^q \right) = \sum_{\pi \in \mathcal{P}_\ell} \prod_{B \in \pi} \kappa_B \left(\tilde{Z}_{k_1}^{(\alpha)}, \dots, \tilde{Z}_{k_\ell}^{(\alpha)} \right)\,,
\end{equation}
where $\tilde{Z}_k^{(\alpha)} = Z_k^{(\alpha)} - \E\left[Z_k^{(\alpha)} \right]$. 
In conjunction with Lemma~\ref{lem:product_formula}, one therefore has the desired convergence of moments
\begin{equation}
    \E \left[ \prod_{r=1}^\ell a_{k_r, n} \right]
    \to 
    \E \left[ \prod_{r=1}^\ell \tilde{Z}_{k_r}^{(\alpha)} \right] \,, \qquad \nto\,,
\end{equation}
which finishes the proof of Proposition \ref{prop:moment_convergence}.
\end{proof}

\section{Limit cases : Poisson to Gaussian behavior}
\label{sec:limit_cases}

\noindent
The goal of this section is to show convergence results in the boundary cases $\alpha \in \{0, 2 \}$. 
We start with the case $\alpha=0$ in Section \ref{subsec:alpha_0}, where Theorem \ref{thm:poisson} and the results on one-hot matrices, that is, Lemma \ref{lem:conv_frobenius} and Proposition \ref{prop:spectral_measure_one_hot} are proved. Section \ref{subsec:gaussian_case} deals with the case $\alpha=2$ and proves Theorem~\ref{thm:gaussian}.

\subsection{Poisson behavior for $\alpha=0$}
\label{subsec:alpha_0}

\begin{proof}[Proof of Theorem~\ref{thm:poisson}]
    In the case where $\alpha=0$, one has $c_k^{(0)} = 1$ for every $k \geqslant 1$. For $\alpha = 0$, the $Beta(1-\alpha/2, \alpha) = \delta_1$ is the Dirac measure at $1$. Therefore, $\mu_q^{(0)} = \delta_1$ and consequently
    \begin{equation*}
        \nu_q^{(0)} = \frac{1}{2q} (\delta_{-1} + \delta_1) \,.
    \end{equation*}
    Thus, the corresponding $q$-Poisson process from Definition \ref{def:q_poisson_process} becomes
    \begin{equation*}
        N_q^{(0)} = \delta_{N_{q}^-} + \delta_{ N_{q}^+}\,,
    \end{equation*}
    where $\big(N_{\ell}^-\big)_{\ell\ge 1}$ and $\big(N_{\ell}^+\big)_{\ell\ge 1}$ are two i.i.d.\ sequences of independent random variables  with $N_{\ell}^-$ follwoing a Poisson $\mathcal{P}\big(\frac{1}{2\ell}\big)$ distribution. Thus, the random variables $V_{q, m}^{(0)}$ from \eqref{eq:def_V_q_m} are 
    \begin{equation*}
        V_{q, m}^{(0)} = q (f_m(-1) N_q^- + f_m(1)N_q^+) = q (N_q^+ + (-1)^m N_q^-) \,,
    \end{equation*}
    so that 
    \begin{equation*}
        Z_k^{(0)} = \sum_{q | k} q (N_{q}^+ + (-1)^{k/q} N_{q}^-) \,, \qquad k\ge 1\,.
    \end{equation*}
     Its expectation is 
    \begin{equation*}
        \E \left[Z_k^{(0)} \right] = \frac{1}{2} \sum_{q | k} (1 + (-1)^{k/q}) =: \tilde{c}_k^{(0)} 
        = \begin{cases}
            0 &\text{ if } k \text{ odd} \\
            \sum_{q | k/2} 1 &\text{ if } k \text{ even}.
        \end{cases} 
    \end{equation*}
    Note that this definition of $\tilde{c}_k^{(0)}$ is consistent with \eqref{eq:def_tilde_ck}. 
    Using the moment-cumulant formula \eqref{eq:moment_cumulant_formula2} 
    gives for $k_1,\ldots,k_\ell\ge 1$
    \begin{equation*}
        \E \left[ \prod_{r=1}^\ell a_{k_r, n} \right]
        \to \E \left[ \prod_{r=1}^\ell \tilde{Z}_{k_r}^{(0)} \right] \,, \qquad \nto\,.
    \end{equation*}
    Therefore, one has the convergence in distribution 
    \begin{equation}
        \left(\Tr \left[ \Y^{k_1}_n \right], \dots, \Tr \left[ \Y^{k_\ell}_n \right] \right) \cid \left(\tilde{Z}_{k_1}^{(0)}, \dots, \tilde{Z}_{k_\ell}^{(0)} \right) + \left(\tilde{c}_{k_1}^{(0)}, \dots, \tilde{c}_{k_\ell}^{(0)} \right) \,, \qquad \nto\,,
    \end{equation} 
    and since $\E \left[Z_k^{(0)} \right] = \tilde{c}_k^{(0)}$,
        \begin{equation}
        \left(\Tr \left[ \Y^{k_1}_n \right], \dots, \Tr \left[ \Y^{k_\ell}_n \right] \right) \cid \left(Z_{k_1}^{(0)}, \dots, Z_{k_\ell}^{(0)} \right) \,,\qquad \nto\,.
    \end{equation}
    Since all arguments in the proof of Proposition~\ref{thm:tightness} remain valid, tightness is again guaranteed by Proposition~\ref{thm:tightness}. Therefore, we have verified the conditions of Lemma \ref{lem:conditions_a_b} which now yields the following convergence in distribution  of characteristic polynomial $p_n$ for the topology of local uniform convergence in $\D$:
    \begin{equation*}
        p_n \cid F_0 \,, \qquad \nto\,,
    \end{equation*}
    where $F_0$ is the random holomorphic function 
    \begin{equation*}
        F_0(z) = \exp \left( - \sum_{k \geqslant 1} \frac{z^k}{k} Z_k^{(0)} \right) .
    \end{equation*}
\end{proof}

\begin{remark}[Product expression and zeroes]
    Using the product expression from  \eqref{eq:product_expression} in the case where $\alpha=0$ so that $c_k^{(0)} = 1$ for every $k \geq 1$, together with Lemma \ref{lem:zeroes_convergence} yield that the zeroes of $p_n$ converge to the ones of $\prod_{l \geq 1} (1 - z^{2\ell})^{1/2\ell}$, which is non-vanishing in $\D$. The zeros of the latter are located on the unit circle, which match the rightmost picture in Figure \ref{fig:eigenvalue-clouds}. To derive the convergence, one should extend the domain of convergence in Theorem \ref{thm:poisson} to a disk of convergence radius greater to one.
\end{remark}

\begin{proof}[Proof of Lemma~\ref{lem:conv_frobenius}]
We compare $\Y_n$ to the one-hot matrix $\Y_n^{(0)}$ with rows $\sign(Y_{ik_i}) \e_{k_i}$, $i\in [n]$. For the Frobenius norm $\|\cdot\|_F$, we get
	\begin{align*}
		\| \Y_n-\Y_n^{(0)} \|_F^2 &=  \sum_{i=1}^n \sum_{t=1}^n \big(Y_{it}-\sign(Y_{ik_i}) \1(t=k_i)\big)^2\\
		&=\sum_{i=1}^n \left( 1- Y_{ik_i}^2 + (Y_{ik_i}-\sign(Y_{ik_i}))^2\right)\\
		&= 2\Big( n-\sum_{i=1}^n |Y_{ik_i}|\Big).
	\end{align*}
	Using Markov's inequality, we deduce for $\vep>0$ that
	\begin{align*}
    \P\left( \| \Y_n-\Y_n^{(0)} \|_F >\vep \right)
    &\le \frac{2}{\vep^2} \sum_{i=1}^n \Big(1-\E\big[|Y_{ik_i}|\big]\Big)\to 0\,, \qquad \nto,
	\end{align*}
    where the last step follows by the assumption \eqref{eq:assumption_L} on $L$. 
\end{proof}

The rest of this section concerns the proof of Proposition~\ref{prop:spectral_measure_one_hot} on one-hot matrices which have the same distribution as $\Y_n^{(0)}$. Equivalently, each row has a $\pm 1$ in only one random column index which is uniform on $\{1, \dots, n\}$ and zeros elsewhere, and where rows are independent.
In the following, let us denote by $\{-1, 1 \} \wr \mathfrak{S}_n$ the set of permutation matrices of size $n \geqslant 1$ where non-zero entries can take values in $\{-1, 1 \}$. We refer to these matrices as \textit{signed permutation matrices.}

\begin{lemma}[Permutation-nilpotent representation]\label{lem:permutation-nilpotent} 
    Let $\Y_n^{(0)}$ be a one-hot matrix. Then, up to conjugation by permutation matrices, $\Y_n^{(0)}$ admits the decomposition in upper triangular form  
    \begin{equation*}
        \Y_n^{(0)} = \begin{pmatrix}
            N & * \\
            0 & P
        \end{pmatrix}\,,
    \end{equation*}
    where $N$ is a nilpotent matrix and $P \in \{-1, 1 \} \wr \mathfrak{S}_r$ is a signed permutation matrix  of size $1 \leqslant r \leqslant n$.
\end{lemma}

\begin{proof}
    Let $\Y_n^{(0)}$ be a one-hot matrix of size $n \geqslant 1$. Let us consider the associated directed graph $G = G(\Y_n^{(0)})$ with vertices $V = [n]$ and having directed edges $E = \{ (i, U_i) \}_{1 \leqslant i \leqslant n}$, where $U_i$ is the column index of the non-zero entry in the row $i$. Note 
      that each vertex has outer degree exactly one. Therefore, each connected component of $G$ consists of a directed cycle of length $q \geqslant 1$, together with some directed trees connected to vertices of the cycle, see Figure \ref{fig:component_one_hot}.
      \begin{figure}
          \centering
        \includegraphics[width=0.3\linewidth]{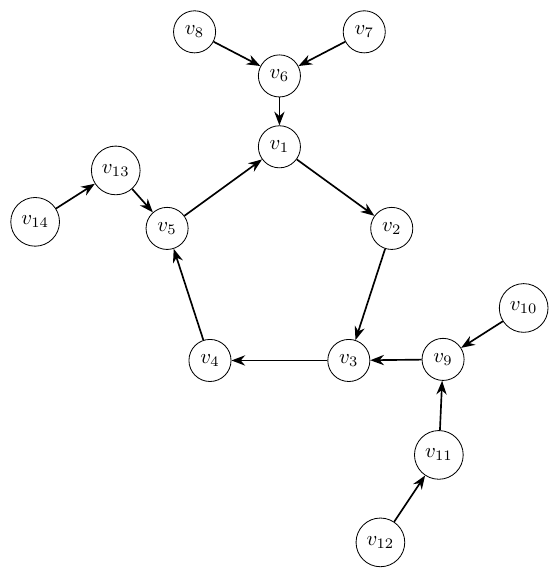}
          \caption{Connected components of the graph induced by a one-hot matrix.}
          \label{fig:component_one_hot}
      \end{figure}
      Ordering the vertices of the cycles with $\{n, n-1, \dots, n-r+1 \}$ yields a block matrix of the form 
    \begin{equation*}
            \begin{pmatrix}
            0 & V_1 & 0 & \cdots & 0\\
            0 & 0 & V_2 & \cdots & 0\\
            \vdots & & & \ddots & \vdots\\
            0 & 0 & 0 & \cdots & V_{q-1}\\
            V_q & 0 & 0 & \cdots & 0
        \end{pmatrix}
    \end{equation*}
    on the bottom right of $\Y_n^{(0)}$.
    We then order the vertices which are in the trees but not in the cycle of the connected component, starting from the cycle, so that the rows of the corresponding vertices are strictly upper triangular. Repeating this operation on every connected component yields the desired decomposition.
\end{proof}

\begin{corollary}[Spectral radius of a one-hot matrix] \label{cor:spectralradius_onehot}
    A one-hot matrix has always an eigenvalue on the unit circle. Therefore, its spectral radius is $1$.
\end{corollary}

\begin{proof}
    Using the decomposition given by Lemma \ref{lem:permutation-nilpotent}, the nonzero eigenvalues may only come from the signed permutation matrix $P$. Since $r \geqslant 1$, $P$ has at least one cycle giving at least one eigenvalue on the unit circle.
\end{proof}

\begin{proposition}[Convergence of traces for one-hot matrices]
  Let $\ell \geqslant 1$ and $k_1, \dots, k_{\ell} \geqslant 1$. Then we have the convergence in 
    \begin{equation}
    \label{eq:conv_traces_one_hot}
    \left(\Tr \left[ (\Y_n^{(0)})^{k_1} \right], \dots, \Tr \left[ (\Y_n^{(0)})^{k_\ell} \right] \right) \cid \left(Z_{k_1}^{(0)}, \dots, Z_{k_\ell}^{(0)} \right) \,, \qquad \nto\,.
    \end{equation}
\end{proposition}

\begin{proof}
We recall Lemma~\ref{lem:allmoments} and let 
$$\bar{\beta}_{2k_1,\ldots,2k_r}:= \E\Big[ \big(y_{11}^{(0)}\big)^{2k_1} \ldots \big(y_{1r}^{(0)}\big)^{2k_r}\Big]\,.$$
Since only one of the $y_{1j}^{(0)}$'s is nonzero, we get for $k\ge 1$ and $r\ge 2$ that
$$n\bar{\beta}_{2k}=1=c_k^{(0)} \quad \text{ and } \quad \bar{\beta}_{2k_1,\ldots,2k_r}=0\,. $$
Setting $\alpha=0$ in the proof of Theorem~\ref{thm:conv_traces} yields the desired result. 
\end{proof}

We are now ready to prove Proposition \ref{prop:spectral_measure_one_hot}.
Using the nilpotent-permutation decomposition $(N_n, P_n)$ of $\Y^{(0)}_n$ from Lemma~\ref{lem:permutation-nilpotent}, one has that 
\begin{equation*}
    \mu_n^{(0)} = \frac{n-c_n}{n} \delta_0 + \frac{1}{n} \sum_{\lambda \in \operatorname{Spec}(P_n)} \delta_{\lambda}
\end{equation*}
where $c_n \ge 1$ is the size of the signed permutation matrix $P_n$ and where $\operatorname{Spec}(P_n)$ denotes the multiset of eigenvalues of $P_n$.
Since only cycle vertices contribute to the size of the signed permutation, using Theorem 2 of \cite{Flajolet_Odlyzko} gives that $\E[c_n] \sim \sqrt{\tfrac{n\pi}{2}}$ as $\nto$. Let $f$ be a continuous and bounded function and let $\vep > 0$. 
Then, using Markov's inequality
\begin{align*}
    \Prob \left[ \left| \int f(x) {\mu_n^{(0)}(\dint x)} - \int f(x) \delta_0( \dint x) \right| \geqslant \vep \right] &= 
    \Prob \left[ \left| \frac{1}{n} \sum_{\lambda \in \operatorname{Spec}(P_n)} f(\lambda) - \tfrac{c_n}{n} f(0) \right| \geqslant \vep \right] \\ 
    & \leqslant \frac{\E [c_n]}{n\vep} ||f||_\infty + \frac{\E[c_n]}{n \vep} |f(0)| \to 0  \,, \qquad \nto\,,
\end{align*}
 which finishes the proof of Proposition \ref{prop:spectral_measure_one_hot}.

We remark that the nonzero eigenvalues of $\Y^{(0)}_n$ are thus outliers of the limiting spectral distribution which is concentrated at zero. However, for every fixed $n \geqslant 1$, one always has at least one eigenvalue lying on the unit circle.
    
\subsection{Domain of attraction of normal distribution: Gaussian behavior}
\label{subsec:gaussian_case}\begin{proof}[Proof of Theorem~\ref{thm:gaussian}]
    Let us assume that $X_{11}\in \DAN$. Then, 
    using Lemma~\ref{lem:allmoments}, one has 
    for every $k \geqslant 2$,
    \begin{equation}
        \label{eq:c_k_to_zero}
        c_k^{(2)} = \lim_{n \to \infty} n \beta_{2k} = 0 \,.
    \end{equation}
    Let us consider a partition $\pi \in \mathcal{P}_\ell$ having a block $B$
    of size $s \geqslant 3$. Then, for every $q \geqslant 1$ such that 
    $q | \gcd(k_j:j\in B)$, one has $\frac{k_B}{q} = \frac{\sum_{j\in B} k_j }{q} \geqslant |B| \geqslant 3$. Therefore, the block gives a vanishing contribution since $c_{k_B/{2q}}^{(2)}  = 0$.
    The right-hand side of \eqref{eq:moments_and_cgd} then becomes a sum over pairings
        \begin{equation}
        \sum_{\pi \in \mathcal{P}_\ell^{(2)}} \prod_{(u, v) \in \pi} \left(  
     \sum_{q | \gcd(k_u, k_v)} 
         q\, \big(c_{(k_u+k_v)/{2q}}^{(2)}\big)^q \right) \,.
    \end{equation}
    As a consequence, the limiting random variables $(\tilde{Z}_k)_{k \geqslant 1}$ are Gaussian. 
    For $k,\ell \ge 1$, we compute the covariances
    \begin{equation}
        \E \left[ \tilde{Z}_{k} \tilde{Z}_\ell \right] = \sum_{q | \gcd(k, \ell)} 
         q\, \big(c_{(k+\ell)/{2q}}^{(2)}\big)^q \,.
    \end{equation}
    If $k \neq \ell$, for any $q | \gcd(k, \ell)$, we have that $(k+\ell)/{q} \geqslant 3$ 
    so that $c_{(k+\ell)/{2q}}^{(2)} = 0$. Therefore, the variables $(\tilde{Z}_k)_{k \geqslant 1}$ are 
    independent. For $k = \ell$, only the term $q=k$ contributes so that $\E |\tilde{Z}_k|^2 = k$
    and therefore $\tilde{Z}_k \sim \mathcal{N}(0, k)$. 
    Since $c_k^{(2)} = 0$ for $k \geq 2$ and $c_1^{(2)}=1$, 
    the deterministic first-order term satisfies $\E\Tr(\Y_n^k)\to \1_{\{2\mid k\}}$. Thus, since the variables $\tilde Z_k$ below describe the centered fluctuations of the traces, one gets that 
    \begin{equation}
    \left(\Tr \left[ \Y^{k_1}_n \right], \dots, \Tr \left[ \Y^{k_\ell}_n \right] \right) \cid \left(\1_{\{2\mid k_1\}}+\sqrt{k_1}Z_{k_1}',\dots,\1_{\{2\mid k_\ell\}}+\sqrt{k_\ell}Z_{k_\ell}'\right) \,,\qquad \nto\,,
    \end{equation}
    where $(Z_k')_{k \geqslant 1}$ are 
    independent, standard Gaussians $\mathcal{N}(0, 1)$.
    Using the product expression \eqref{eq:product_expression} yields that the limiting characteristic polynomial is given by
    \begin{equation*}
    p_n(z) \cid  
    \sqrt{1 - z^{2}} \cdot \exp \left( - \sum_{k \geqslant 1} 
    \frac{z^k}{\sqrt{k}} Z_k' \right) \,.
    \end{equation*}
    \end{proof}

\bibliography{library}
\end{document}